\documentclass[11pt]{amsart}
\usepackage{}

\usepackage[alphabetic]{amsrefs}
\usepackage{amscd}
\usepackage{amsmath}
\usepackage{amssymb}
\usepackage{amsthm}
\usepackage{bigdelim}
\usepackage{color}
\usepackage{enumerate}
\usepackage{mathrsfs}
\usepackage{multirow}
\usepackage[all]{xy}

\def\O{{\mathcal{O}}}

\def\rk{{\mathrm{rank}}}

\theoremstyle{plain}
\newtheorem{thm}{Theorem}[section]
\newtheorem{cor}[thm]{Corollary}
\newtheorem{prop}[thm]{Proposition}

\newtheorem{lem}[thm]{Lemma}
\theoremstyle{definition}
\newtheorem{defn}[thm]{Definition}

\theoremstyle{remark}
\newtheorem{rem}[thm]{Remark}

\begin{document}

\title[Pluricanonical systems on threefolds in positive characteristic]{Frobenius stable pluricanonical systems on threefolds of general type in positive characteristic}

\address{Lei Zhang\\ School of Mathematical Science\\University of Science and Technology of China\\Hefei 230026, P.R.China.}
\email{zhlei18@ustc.edu.cn, zhleimath@163.com}
\author{Lei Zhang}
\maketitle

\begin{abstract}
This paper aims to investigate effectivity problems of pluricanonical systems on varieties of general type in positive characteristic. In practice, we will consider a sub-linear system $|S^0_{-}(X, K_X + nK_X)| \subseteq |H^0(X, K_X +nK_X)|$ generated by certain Frobenius stable sections, and prove that for a minimal terminal threefold $X$ of general type with either $q(X)>0$ or Gorenstein singularities, if $n\geq 28$ then $|S^0_{-}(X, K_X + nK_X)| \neq \emptyset$; if $n\geq 42$ then  the linear system $|S^0_{-}(X, K_X + nK_X)|$ defines a birational map.

\end{abstract}

\section{Introduction}
\textbf{Conventions:} (1) A \emph{fibration} means a projective morphism $f:X \to Y$ of varieties such that $f_*\mathcal{O}_X = \mathcal{O}_Y$. Throughout this paper, as $Y$ frequently appears as the base of a fibration, we use $\eta$ and $\bar{\eta}$ to denote the generic point and geometric generic point of $Y$ respectively, and naturally use $X_{\eta}$ and $X_{\bar{\eta}}$ to denote the  generic fiber and geometric generic fiber of $f$. For a scheme $Z$ we use $Z^{\mathrm{red}}$ to denote the scheme with the reduced structure of $Z$.
\smallskip

(2) For a morphism $\sigma: Z \to X$, if $D$ is a divisor on $X$, especially when $Z$ is birational to a subvariety of $X$, we often use $D|_Z$ to denote the pullback $\sigma^*D$ for simplicity.
\smallskip

(3) Let $D$ be a Weil divisor on a normal variety $X$ (or a Cartier divisor on an integral variety). We regard $\mathcal{O}_X(D)$ as a subsheaf of the constant sheaf of function field $K(X)$ of $X$ via
$$\mathcal{O}_X(D)_x:=\{f \in K(X)| (\mathrm{div}(f) + D)|_U \geq 0 ~\mathrm{for~some~open~set}~U~\mathrm{containing}~x\}.$$
Let $V \subseteq H^0(X, \mathcal{O}_X(D))$ be a finite dimensional linear subspace. If the linear system $|V|$ defines a birational map, we will simply say that $V$ or $|V|$ is birational.

\smallskip

(4) For two integral divisors $D_1,D_2$ on a normal variety $X$ such that $D_1 \leq D_2$, let $E= D_2-D_1$. We will use $s_E$ to denote a nonzero section of $\mathcal{O}_X(E)$ with zero $E$ (unique up to multiplying with a nonzero constant if $X$ is a projective variety). Then we have a natural inclusion $H^0(X, \mathcal{O}_X(D_1)) \otimes s_E \subseteq H^0(X, \mathcal{O}_X(D_2))$. When $X$ is a projective variety, the natural inclusion map $H^0(X, \mathcal{O}_X(D_1)) \hookrightarrow H^0(X, \mathcal{O}_X(D_2))$ means the map induced by tensoring with $s_E$.
\bigskip

Pluricanonical system $|nK_X|$ plays an important role in the classification of varieties.
For the class of varieties with non-negative Kodaira dimension, it is significant to get a lower bound of  $n$ such that
\begin{itemize}
\item
the linear system $|nK_X|\neq \emptyset$ (effective nonvanishing problem) and
\item
the $n$-canonical map defined by $|nK_X|$ is birationally equivalent to the Iitaka fibration (effective Iitaka fibration problem).
\end{itemize}
Here we are only concerned with varieties of general type. For a smooth projective surface $X$ of general type over an algebraically closed field of arbitrary characteristic, it is known that $|5K_X|$ is birational (\cite{SB91,Re88}). In general, over the field of complex numbers $\mathbb{C}$, there exists a number $M(d)$ such that, for any $d$-dimensional smooth projective varieties of general type, if $m \geq M(d)$ then $|mK_X|$ is birational (\cite{HM06, Tak06}), and for threefolds we may take $M(3) = 126$ (\cite{CC10a,CC10b}).

This paper aims to investigate pluricanonical systems of varieties of general type in positive characteristic. As is known to experts, Kawamata-Viehweg vanishing is a key technical tool in the study of adjoint linear system $|K_X + L|$ in characteristic zero, which enables one to extend a section on the log canonical center to the whole variety. Unfortunately, in positive characteristic, Kawamata-Viehweg vanishing fails for some varieties (\cite{Ra78}). To substitute for the role of this vanishing in positive characteristic, the idea is to combine Fujita vanishing (\cite{Fuj83, Ke03}) with Mumford regularity, then one can apply Frobenius amplitude to produce certain global sections of $K_X + L$ (\cite{Ke08, Sch14}). These sections are called \emph{Frobenius stable sections}, and the corresponding sub-linear system is called a \emph{Frobenius stable adjoint linear system} (Sec. \ref{sec:Frobenius stable sections}). Following this idea, we shall adapt some classical inductive approaches from characteristic zero to positive characteristic case. We can apply them to treat varieties endowed with certain fibrations, which arise either from pluricanonical maps or from Albanese morphisms. With the help of Riemann-Roch formula, we are able to prove effective nonvanishing and birationality of Frobenius stable pluricanonical systems on lower dimensional varieties.

\subsection{Effectivity for Linear systems on varieties equipped with certain fibrations}
We briefly recall a classical inductive strategy from characteristic zero as follows. For a smooth projective variety $X$ over an algebraically closed field of characteristic zero, if given a natural number $n_1$ such that $\dim |n_1K_X| \geq 1$, which induces a rational map $f:X \dashrightarrow Y$ with generic fiber $F$, and given a number $n_2$ such that $|n_2K_F|$ defines birational map of $F$, then one can get a suitably bigger number $M(n_1, n_2)$ such that for $m\geq M(n_1, n_2)$ the linear system $|mK_X|$ is  birational (\cite{Ch04}). The most important step to carry out this strategy is to extend sections on a fiber to the whole variety, hence one needs vanishing results and weak positivity of the pushforward of (relative) pluricanonical sheaves (\cite{Ko86}). This approach still works if we consider Frobenius stable adjoint linear systems. We will use the following criterion, which can be seen as a generalization of Keeler's result \cite{Ke08} to higher relative dimensional cases.
\begin{thm}[=Theorem \ref{thm:bir-criterion-for-induction}]\label{thm:intr-bir-crt1}
Let $f:X \to Y$ be a fibration of normal projective varieties over an algebraically closed field $k$ of characteristic $p$, and let $d = \dim Y$.
Let $D$ be a nef and big $\mathbb{Q}$-Cartier $\mathbb{Q}$-divisor on $X$, and $H, \tilde{H}$ two $\mathbb{Q}$-Cartier Weil divisors on $Y$ such that $|H|$ defines a generically finite map and $|\tilde{H}|$ is birational.

(i) If $S_{-}^0(X_{\eta}, K_{X_{\eta}} + D|_{X_{\eta}}) \neq 0$ then $S_{-}^0(X, K_X +D + f^*sH)\neq 0$ for any $s \geq d$.

(ii) If $S_{-}^0(X_{\eta}, K_{X_{\eta}} + D|_{X_{\eta}})$ is birational then $S_{-}^0(X, K_X +D + f^*sH)$ is birational for $s \geq d+1$; and if moreover $S_{-}^0(X, K_{X} +D + df^*H - f^*\tilde{H}) \neq 0$ then $S_{-}^0(X, K_X +D + f^*dH)$ is birational.
\end{thm}

About the notions and the assumptions in the theorem above, we make the following remarks. First, we do not require $D$ to be integral to keep certain flexibility in the application. Second, the notion $S_{-}^0(X, K_X + D)$ is a subspace of the space of Frobenius stable sections $S^0(X, K_X + \ulcorner D\urcorner)$ (Section \ref{sec:Frobenius stable sections}). The advantage of this subspace lies in that each section of this type on the generic fiber can be lifted to a global one.
\medskip

Inspired by the idea of continuous global generation (CGG) introduced by Pareschi and Popa \cite{PP03}, we can prove the following theorem, which is used to treat the case of irregular varieties.
\begin{thm}[=Theorem \ref{thm:bir-criterion-irr}]\label{thm:intr-bir-crt2}
Let $X$ be a smooth projective variety  over an algebraically closed field $k$ of characteristic $p$, and let $a: X \to A$ be a morphism to an abelian variety. Denote by $f: X \to Y$ the fibration arising from the Stein factorization of $a: X \to A$. Let $D, D_1,D_2$ be three divisors on $X$. Assume that $D$ is nef, big and $\mathbb{Q}$-Cartier.

(i) If $S^0_{-}(X_{\eta}, K_{X_{\eta}} + D_{\eta}) \neq 0$, then for any $\mathcal{P}_{\alpha} \in \mathrm{Pic}^0(A)$, $H^0(X, K_X + \ulcorner D\urcorner + a^*\mathcal{P}_{\alpha})) \neq 0$, and there exists some $\mathcal{P}_{\beta} \in \mathrm{Pic}^0(A)$ such that $S_{-}^0(X, K_X + \ulcorner D\urcorner + a^*\mathcal{P}_{\beta})\neq 0$.

(ii) Assume that $S^0_{-}(X_{\eta}, K_{X_{\eta}} + D_{\eta}) \neq 0$, $D_1$ is integral and for any $\mathcal{P}_{\alpha} \in \mathrm{Pic}^0(A)$, $|D_1 + a^*\mathcal{P}_{\alpha}| \neq \emptyset$. Then for any $\mathcal{P}_{\alpha_0} \in \mathrm{Pic}^0(A)$, $S^0_{-}(X, K_X + D + D_1 + a^*\mathcal{P}_{\alpha_0}) \neq 0$.

(iii) Assume that $S^0_{-}(X_{\eta}, K_{X_{\eta}} + D_{\eta})$ is birational, both $D_1$ and $D_2$ are integral and for any $\mathcal{P}_{\alpha} \in \mathrm{Pic}^0(A)$, $|D_i + a^*\mathcal{P}_{\alpha}| \neq \emptyset$. Then for any $\mathcal{P}_{\alpha_0} \in \mathrm{Pic}^0(A)$, $S^0_{-}(X, K_X + D + D_1+D_2 + a^*\mathcal{P}_{\alpha_0})$ is birational.

(iv) Assume that $S^0_{-}(X_{\eta}, K_{X_{\eta}} + D_{\eta})$ is birational, and $D_1,D_2$ are nef and big $\mathbb{Q}$-Cartier $\mathbb{Q}$-divisors such that $S^0_{-}(X_{\eta}, K_{X_{\eta}} + (D_i)_{\eta}) \neq 0$. Then for any $\mathcal{P}_{\alpha_0} \in \mathrm{Pic}^0(A)$,
$S^0_{-}(X, K_X + D + (K_{X} + \ulcorner D_1\urcorner) + (K_{X} +  \ulcorner D_2\urcorner )+ a^*\mathcal{P}_{\alpha_0}))$ is birational.
\end{thm}

\subsection{Effectivity of pluricanonical maps of lower dimensional varieties} We explain our strategy as follows. Without loss of generality we consider a minimal threefold $X$ of general type and separate two cases $q(X) =0$ and $q(X) >0$.

(1) For the first case $q(X) = 0$, we shall apply the first criterion (Theorem 1.1). We are left to find a number $n_0$ such that $\dim|n_0K_X| \geq 1$. If $p_g(X) >1$ we may take $n_0 =1$. If $p_g(X)=0$ then $\chi(\mathcal{O}_X) \geq 0$. For this case we first apply Riemann-Roch formula to find a number $n_0$ such that $\chi(X, n_0K_X) \geq 2$, here we need to prove a Miyaoka-Yau type inequality (Theorem \ref{thm:my-ineq}). Then we intend to show $h^2(X, n_0K_X) = 0$ by using some results of bend-and-break, but we have to require that $X$ has Gorenstein singularities for an unhappy technical reason (Lemma \ref{lem:van-h2}).

(2) For the second case $q(X) >0$, as $X$ has nontrivial Albanese map, we can apply Theorem 1.2 and argue according to the relative Albanese dimension.
\smallskip

It is worth mentioning that, to do induction, we need to verify the effectivity conditions of Frobenius stable pluricanonical systems $|S^0_{-}(X_{\eta}, K_{X_{\eta}} + nK_{X_{\eta}})|$ on the generic fiber, but in practice, thanks to Theorem \ref{thm:bir-geo-gen}, we may pass to a normal model $Z_{\bar{\eta}}$ of $X^{\mathrm{red}}_{\bar{\eta}}$ and only need to verify the corresponding conditions for $|S^0_{-}(Z_{\bar{\eta}}, K_{Z_{\bar{\eta}}} + nK_{X_{\eta}}|_{Z_{\bar{\eta}}})|$. It is convenient to work with $Z_{\bar{\eta}}$ since it is defined over an algebraically closed field.
In Theorem \ref{thm:eff-curve} and \ref{thm:eff-surface}, we first obtain effectivity results for Frobenius stable adjoint linear systems on curves and surfaces. In particular, we prove that for smooth projective curves of general type, $S^0_{-}(X, K_X + nK_X) \neq 0$ for $n\geq 1$, and $S^0_{-}(X, K_X + nK_X)$ is very ample if $n\geq 2$; and for surfaces of general type, the two lower bounds are $4$ and $7$ respectively. With these preparations, we finally prove the following theorem for threefolds.



\begin{thm}\label{thm:eff-3folds}
Let $X$ be a minimal terminal threefold of general type over an algebraically closed field of characteristic $p$.

(1) Assume $q(X) >0$. Then $S^0_{-}(X, K_X + nK_X) \neq 0$ if $n\geq 11$, and $S^0_{-}(X, K_X + nK_X)$ is birational if $n\geq 21$; and if moreover $p>2$, then $S^0_{-}(X, K_X + nK_X) \neq 0$ if $n\geq 9$, and $S^0_{-}(X, K_X + nK_X)$ is birational if $n\geq 17$.

(2) Assume $q(X) =0$ and $X$ has only Gorenstein singularities. Set $n_0(2)=13, n_0(3) =10, n_0(5) =9$, $n_0(7) = 8$  if $p \geq 7$. Then
$S^0_{-}(X, K_X + nK_X) \neq 0$ if $n\geq 2n_0(p) + 2$ and $S^0_{-}(X, K_X + nK_X)$ is birational if $n\geq 3n_0(p) + 3$.
\end{thm}

\subsection{Further remarks and questions} $\empty$ \par

(1) For surfaces, the birational lower bound for Frobenius stable pluricanonical systems is very near to the classical one (for pluricanonical systems), it is possibly optimal. For threefolds of general type, in characteristic zero we have that $|5K_X|$ is birational if $X$ either has $q(X) >0$ or has a Gorenstein minimal model (\cite{CH07, CCZ07}), comparing with this result, the bound obtained in Theorem \ref{thm:eff-3folds} seems far from optimal.
\smallskip

(2) Our proof relies on two technical assumptions.

The first one is the existence of minimal models. When this paper is in preparation, (log) minimal model theory in dimension two has been established (\cite{Tan14, Tan18a, Tan20}), and existence of minimal models in dimension three has been proved when the characteristic $p \geq 5$ (\cite{HX15,Bir16,HW19}). Moreover when $p>5$, for minimal threefolds, when $q(X) >0$ abundance has been proved (\cite{Zha17}), and  when $q(X) =0$ only nonvanishing has been proved (\cite{XZ19}).

The second one is the Gorenstein condition on the minimal model when $q(X) =0$. In fact, by our proof, if the minimal model $X$ has rational singularities we can get a bound relying on the Cartier index of $K_X$. It is proved in \cite{ABL20} that a terminal singularity over an algebraically closed field of characteristic $p >5$ is rational, but when $p\leq 5$ this is not necessarily true.

It is expected that there exists a birational lower bound only depending on the volume of $K_X$. But we cannot drop the above two assumptions.

\smallskip

(3) Furthermore, effectivity problems for varieties of intermediate Kodaira dimension are also of great significance. There have been many progresses in characteristic zero, and we refer the reader to \cite{BZ16} for the recent results and techniques. But in characteristic $p$, up to now we do not have any result for threefolds, because of the lack of some deep results from Hodge theory.
\medskip

This paper is organized as follows. In Section \ref{sec:pre} we introduce the notion of Frobenius stable adjoint linear system and study its behaviour under base changes. In Section \ref{sec:main-bir-crt}, we prove the two effectivity criteria in Theorem \ref{thm:intr-bir-crt1} and \ref{thm:intr-bir-crt2}. In Section \ref{sec:dim2} we investigate effectivity of Frobenius stable adjoint linear systems on curves and surfaces. In Section \ref{sec:my-ineq} we prove a Miyaoka-Yau type inequality. In Section \ref{sec:dim3}, we study the Frobenius stable pluricanonical systems of threefolds and prove Theorem \ref{thm:eff-3folds}.
\medskip

{\small \noindent\textit{Acknowledgments.}
The author thanks Prof. Meng Chen and Chen Jiang for useful discussions. This research is partially supported by grant
NSFC (No. 11771260), the Fundamental
Research Funds for Central Universities and the project ``Analysis and Geometry on Bundles'' of Ministry of Science and Technology of the People's Republic of China. }

\section{Preliminaries}\label{sec:pre}
In this section, we will collect some useful technical results and introduce the notion of Frobenius stable adjoint linear system.

\subsection{Fujita vanishing} The key technical tool of this paper is Fujita vanishing proposed by Fujita \cite{Fuj83}. Here we will present a generalized version due to Keeler \cite{Ke03}.

\begin{thm}[{\cite[Theorem 1.5]{Ke03}}]
Let $f: X \rightarrow Y$ be a projective morphism over a Noetherian scheme, $H$ an $f$-ample line bundle and $\mathcal{F}$ a coherent sheaf on $X$. Then there exists a positive integer $N$ such that, for every $n >N$ and every relatively nef line bundle $L$ on $X$
$$R^if_*(\mathcal{F}\otimes H^n \otimes L) = 0~ \mathrm{~if~} i>0.$$
\end{thm}

As a corollary, we get the following version which will be used frequently in the proof.
\begin{thm}\label{thm:var-fujita-vanishing}
Let $f: X \rightarrow Y$ be a projective morphism of normal varieties, let $H$ be an $f$-ample $\mathbb{Q}$-Cartier $\mathbb{Q}$-divisor and $D$ another divisor on $X$, and let $\mathcal{F}$ be a coherent sheaf on $X$. Then there exists a positive integer $N$ such that, for every $n >N$ and every relatively nef line bundle $L$ on $X$
$$R^if_*(\mathcal{F}\otimes \mathcal{O}_X(\ulcorner nH + D\urcorner) \otimes L) = 0~ \mathrm{~if~} i>0.$$
\end{thm}

\subsection{Frobenius stable sections}\label{sec:Frobenius stable sections} Let $K$ be an $F$-finite field of characteristic $p>0$. Let $X$ be a normal projective scheme over $K$ of finite type. Let $D$ be a $\mathbb{Q}$-divisor on $X$. Let $F^e: X^e \to X$ denote the $e$-th iteration of absolute Frobenius map of $X$, which is a finite morphism since $K$ is assumed to be $F$-finite. By the duality theory we have the trace map
$Tr^e_X: F^e_*\mathcal{O}_X(K_X) \to \mathcal{O}_X(K_X)$, and after tensoring this map with $\mathcal{O}_X(\ulcorner D\urcorner)$ and taking saturation, the trace map induces a homomorphism of $\mathcal{O}_X$-modules
$$F^{e}_*\mathcal{O}_X(K_X+ p^e\ulcorner D\urcorner)\to \mathcal{O}_X(K_X + \ulcorner D\urcorner).$$
We will use simply $Tr^e_X$ or $Tr^e$, if no confusion occurs, to denote various maps induced by the trace map of $F^e$.
By $\ulcorner p^eD\urcorner \leq p^e\ulcorner D\urcorner$, we have a natural inclusion map $\mathcal{O}_X(K_X+ \ulcorner p^eD\urcorner) \hookrightarrow \mathcal{O}_X(K_X+ p^e\ulcorner D\urcorner)$. Consider the composition map
$$Tr^e: F^{e}_*\mathcal{O}_X(K_X+ \ulcorner p^eD\urcorner) \hookrightarrow F^{e}_*\mathcal{O}_X(K_X+ p^e\ulcorner D\urcorner)\to \mathcal{O}_X(K_X + \ulcorner D\urcorner)$$
and denote
$$S^e(X, K_X + D)= \mathrm{Im}(Tr^e:H^0(X, F^{e}_*\mathcal{O}_X(K_X+ \ulcorner p^eD\urcorner)) \to H^0(X, \mathcal{O}_X(K_X+ \ulcorner D\urcorner))).$$
We have the following factorization
\begin{align*}
Tr^e: F^{e}_*\mathcal{O}_X(K_X+ \ulcorner p^eD\urcorner) &\hookrightarrow F^{e}_*\mathcal{O}_X(K_X+ p\ulcorner p^{e-1}D\urcorner) \\
&\xrightarrow{F^{e-1}_*Tr^1}F^{e-1}_*\mathcal{O}_X(K_X+ \ulcorner p^{e-1}D\urcorner) \xrightarrow{Tr^{e-1}} \mathcal{O}_X(K_X + \ulcorner D\urcorner),
\end{align*}
and then get a natural inclusion $S^e(X, K_X + D) \subseteq S^{e-1}(X, K_X + D)$.
The set of Frobenius stable sections is defined as follows, which forms a $K$-linear space
$$S^0(X, K_X + D)= \bigcap_{e>0}S^e(X, K_X + D).$$
If $X$ is regular and $D$ is an integral divisor, then this notion coincides with the usual one (\cite[Sec. 3]{Sch14}).

Take another divisor $D' \leq D$. Set $E= \ulcorner D\urcorner - \ulcorner D'\urcorner$. Then we have the following commutative diagram
{\small $$\xymatrix{&F^{e}_*\mathcal{O}_X(K_X+ \ulcorner p^eD'\urcorner) \ar@{^(_->}[r]\ar[d]^{Tr^e} &F^{e}_*\mathcal{O}_X(K_X+ \ulcorner p^eD\urcorner)\ar[d]^{Tr^e}\ar@{^(_->}[r] &F^{e}_*\mathcal{O}_X(K_X+  p^e(\ulcorner D'\urcorner + E))\ar[d]^{Tr^e}  \\
&\mathcal{O}_X(K_X+ \ulcorner D'\urcorner)) \ar[r]^{\otimes s_E} &\mathcal{O}_X(K_X+ \ulcorner D\urcorner) \ar@{=}[r] &\mathcal{O}_X(K_X+ \ulcorner D'\urcorner + E)}$$}
and conclude $S^e(X, K_X + D')\otimes s_E \subseteq S^e(X, K_X + D)$, which is compatible with the natural inclusion $H^0(X, K_X + \ulcorner D'\urcorner)\otimes s_E \subseteq H^0(X, K_X + \ulcorner D\urcorner)$. In turn, we obtain that
$$S^0(X, K_X + D')\otimes s_E \subseteq S^0(X, K_X + D).$$
\smallskip

Let $\Delta$ be an effective $\mathbb{Q}$-divisor. We define
$$S_{\Delta}^e(X, K_X + D):=S^e(X, K_X + D-\Delta)\otimes s_E \subseteq S^e(X, K_X + D)$$
where $E = \ulcorner D\urcorner - \ulcorner D-\Delta\urcorner$ and denote
\begin{align*}
S_{\Delta}^0(X, K_X + D)=\bigcap_{e\geq0}S_{\Delta}^e(X, K_X + D).
\end{align*}
Obviously for another effective divisor $\Delta'$, if $\Delta' \leq \Delta$ then $S_{\Delta}^0(X, K_X + D) \subseteq S_{\Delta'}^0(X, K_X + D)$.

To extend a section on certain closed subvariety to the whole variety, it is convenient to deal with ample divisors. To this end, we introduce a smaller subspace $S_{-}^0(X, K_X + D) \subseteq S^0(X, K_X + D)$. First assume $D$ is a nef and big $\mathbb{Q}$-Cartier $\mathbb{Q}$-divisor, and let $\Theta_{D}^{\mathrm{amp}}$ denote the set of effective $\mathbb{Q}$-divisors $\Delta$ such that $D- \Delta$ is an ample $\mathbb{Q}$-Cartier $\mathbb{Q}$-divisor, then define this subspace as follows
\begin{equation}\label{def:S0}
S_{-}^0(X, K_X + D)= \bigcap_{\Delta \in \Theta_{D}^{\mathrm{amp}}}~(\bigcup_{t\in \mathbb{Q}^+}S_{t\Delta}^0(X, K_X + D)~\subseteq S^0(X, K_X + D)).
\end{equation}
In general, we let $\mathcal{B}_{D}^{\mathrm{nef}}$ be the set of effective divisors $B$ such that $D-B$ is a nef and big $\mathbb{Q}$-Cartier $\mathbb{Q}$-divisor, and define
\begin{equation}\label{def:gS}S_{-}^0(X, K_X + D)= \mathrm{Im}(\sum\limits_{B \in \mathcal{B}_{D}^{\mathrm{nef}}}S_{-}^0(X, K_X + D-B) \hookrightarrow S^0(X, K_X + D)).\end{equation}
In this paper we call the linear system generated by $S_{-}^0(X, K_X + D)$ a \emph{Frobenius stable adjoint linear system}.
\smallskip

\begin{prop}\label{prop:F-stable-section} Let $X,D$ be as above.

(i) For an effective Weil divisor $E$, we have
$$S_{-}^0(X, K_X + D)\otimes s_E \subseteq S_{-}^0(X, K_X + D + E).$$

(ii) Assume $D$ is a nef and big $\mathbb{Q}$-Cartier $\mathbb{Q}$-divisor on $X$. Then there exists a closed subvariety $T \subsetneq X$ such that for any $\Delta \in \Theta_{D}^{\mathrm{amp}}$, if $\mathrm{Supp}~ \Delta\supseteq T$ then
$$S_{-}^0(X, K_X + D) = S_{t\Delta}^0(X, K_X + D)$$
for   sufficiently small $t >0$.

(iii) Assume that $X$ is regular and $D$ is a $\mathbb{Q}$-divisor on $X$.

 (iii-1) If $D$ is ample and there is an integer $a>0$ such that $p^aD$ is integral, then for any nef line bundle $L$ on $X$, there exists a number $N$ independent of $L$ such that for any $e\geq N$
$$S_{-}^0(X, K_X + D +      L)=S^0(X, K_X + D + L) = S^{ea}(X, K_X + D + L).$$

(iii-2) If $D$ is nef and big then there exists $\Delta \in \Theta_{D}^{\mathrm{amp}}$ such that, for any nef line bundle $L$ on $X$ and any $t \in \mathbb{Q}^{+}$
$$S_{t\Delta}^0(X, K_X + D + L) \subseteq S_{-}^0(X, K_X + D + L),$$
and as a consequence the equality is attained if $t$ is sufficiently small (not necessarily independent of $L$).

(iv) Assume $D$ is a nef and big $\mathbb{Q}$-Cartier $\mathbb{Q}$-divisor on $X$. Let $\sigma: X' \to X$ be a birational morphism from another normal projective variety $X'$ over $K$ and let $E$ be an effective $\sigma$-exceptional $\mathbb{Q}$-Cartier $\mathbb{Q}$-divisor on $X'$. Then

(iv-1) the trace map of $\sigma$ induces a natural injective map $Tr_{\sigma}: S_{-}^0(X', K_{X'} + \sigma^*D) \to S_{-}^0(X, K_X + D)$;

(iv-2) if $X$ is $\mathbb{Q}$-factorial and $E' = \ulcorner \sigma^*D + E\urcorner - \ulcorner \sigma^*D\urcorner$, then $S_{-}^0(X', K_{X'} + \sigma^*D)\otimes s_{E'} = S_{-}^0(X', K_{X'} + \sigma^*D + E)$.

\end{prop}
\begin{proof}
(i) By the construction we have that for $\Delta \geq 0$, $S_{\Delta}^0(X, K_X + D) \otimes s_E = S_{\Delta + E}^0(X, K_X + D + E)$.  From this fact, the assertion (i) follows from the definition.

\smallskip

(ii) Assume that $D$ is nef and big. Observe that
\begin{itemize}
\item
for $\Delta_1, \Delta_2 \in \Theta_{D}^{\mathrm{amp}}$, if $\mathrm{Supp}~\Delta_1 \subseteq \mathrm{Supp}~\Delta_2$ then
$$\bigcup_{t\in \mathbb{Q}^+}S_{t\Delta_1}^0(X, K_X + D) \supseteq \bigcup_{t\in \mathbb{Q}^+}S_{t\Delta_2}^0(X, K_X + D);$$
\item
since $S_{-}^0(X, K_X + D)$ is a finite dimensional $k$-linear space, in the definition (\ref{def:S0}) the equality can be attained by taking finitely many divisors $\Delta_1, \cdots, \Delta_r \in \Theta_{D}^{\mathrm{amp}}$.
\end{itemize}
Therefore, $T = \mathrm{Supp}~\sum_{i=1}^r \Delta_i$ satisfies our requirement.

\smallskip
(iii) First we apply (iii-1) to show (iii-2). Fix an effective divisor $\Delta$ such that $D - \Delta$ is ample and  $\mathrm{Supp}~D \subseteq \mathrm{Supp}~\Delta$. Let $t_n= \frac{1}{n}$. We can get a sequence of $\mathbb{Q}$-divisors $\Delta_n$ such that $t_{n+1}\Delta \leq \Delta_n \leq t_n\Delta$, $D- \Delta_n$ are ample and the indices of $D- \Delta_n$ are powers of $p$.
Then it follows that
\begin{align*}
&S_{t_n\Delta}^0(X, K_X + L+ D)\cong S^0(X, K_X + L+ D - t_n\Delta)\\
 &\hookrightarrow S^0(X, K_X + L+D - \Delta_n) = S_{-}^0(X, K_X +L+ D - \Delta_n) \hookrightarrow S_{-}^0(X, K_X + L + D)
\end{align*}
where the equality on the second row is due to (iii-1). From this, we can conclude (iii-2).

We start to prove (iii-1). Let $T$ be a closed subvariety as in (ii).
Take an effective $\mathbb{Q}$-divisor $\Delta$ with $\mathrm{Supp}~T \subseteq \mathrm{Supp}~\Delta$ and coefficients being indivisible by $p$. Since $X$ is regular, we may replace $\Delta$ with some small multiple such that, there exists a large positive integer $g$ such that $(p^g-1)\Delta$ is integral, and that the trace map $Tr_{\Delta}^g: F^{g}_*\mathcal{O}_X((1-p^g)(K_X + \Delta)) \to \mathcal{O}_X$ is surjective. Let $\mathcal{K}_g$ denote the kernel of $Tr_{\Delta}^g$. Then we have the following exact sequence
$$0 \to \mathcal{K}_g \to F^{g}_*\mathcal{O}_X((1-p^g)(K_X + \Delta)) \to \mathcal{O}_X \to 0.$$
By the assumption we may also assume $p^gD$ is integral. For natural numbers $e,s$, we can deduce the following exact sequences
\begin{align*}(*)_{e,s}:~0 \to F^{(e+s)g}_*(\mathcal{K}_g\otimes \mathcal{O}_X&(K_X+ p^{(e+s)g}(D+L))-(p^{sg}-1)\Delta)) \\
\to  F^{(e + s +1)g}_*&\mathcal{O}_X(K_X+ p^{(s+1)g} p^{eg}(D+L) - (p^{(s+1)g}-1)\Delta) \\
&\to F^{(e+s)g}_*\mathcal{O}_X(K_X+ p^{sg} p^{eg}(D+L) -(p^{sg}-1)\Delta) \to 0.
\end{align*}
Since both $D$ and $D-\Delta$ are ample and $L$ is nef, applying Fujita vanishing (Theorem \ref{thm:var-fujita-vanishing}), we can show that
there exists some $e_0>0$ (only dependent on $D$ and $\Delta$) such that for any integer $s\geq 0$,
$$H^1(X, F^{(e_0+s)g}_*(\mathcal{K}_g\otimes \mathcal{O}_X(K_X+  p^{(e_0+s)g}(D+L)-(p^{sg}-1)\Delta)))=0.$$
Fix such an $e_0$ and take the cohomology of $(*)_{e_0, s}$ for $s\geq0$. By induction on $s$, we can show for each $s > 0$ the trace map
\begin{align*}
\eta_{e_0, s}: H^0(X, F^{(e_0+s)g}_*\mathcal{O}_X(K_X+ p^{sg} &p^{e_0g}(D+L) - (p^{sg}-1)\Delta))) \\
 &\to H^0(X, F^{e_0g}_*\mathcal{O}_X(K_X+  p^{e_0g}(D+L)))
\end{align*}
is surjective.
From this we conclude that for any $s\geq 0$,
$$S^{(e_0+s)g}(X, K_X+D+L) =  S^{e_0g}(X, K_X+D+L) =  S^{0}(X, K_X+D+L).$$
Since $p^gD$ is integral, we can take a sufficiently small $t>0$ such that for any integer $s>0$,
$$
\ulcorner p^{(s+e_0)g}(D +L - t\Delta)\urcorner - (p^{sg}p^{e_0g}(D+L) - (p^{sg}-1)\Delta) = \ulcorner(p^{sg}-1-tp^{(s+e_0)g})\Delta \urcorner>0
$$
which gives a natural inclusion
$$\mathcal{O}_X(K_X+ p^{sg} p^{e_0g}(D+L) - (p^{sg}-1)\Delta) \hookrightarrow \mathcal{O}_X(K_X+ \ulcorner p^{(s+e_0)g}((D+L) - t\Delta)\urcorner.$$
From the surjectivity of $\eta_{e_0, s}$, we see that the following trace map is surjective
$$H^0(X, F^{(e_0+s)g}_*\mathcal{O}_X(K_X+ \ulcorner p^{(s+e_0)g}((D+L) - t\Delta)\urcorner)) \to S^0(X, K_X+D+L).$$
Therefore, by the choice of $\Delta$ it follows that $S_{-}^0(X, K_X+D+L) = S^0(X, K_X+D+L)$.

To compute $S^0(X, K_X+D+L)$, we set $\Delta =0$ and $g=a$ in $(*)_{e,s}$. Then by Fujita vanishing there exists a number $N$ independent of $L$ such that for any integer $s\geq 0$
$$H^1(X, F^{(N+s)a}_*(\mathcal{K}_g\otimes \mathcal{O}_X(K_X+  p^{(N+s)g}(D+L))))=0.$$
We can show that $S^0(X, K_X+D+L) =  S^{Na}(X, K_X+D+L)$ by similar argument of the previous paragraph.

\smallskip

(iv) We consider the following commutative diagram
$$\xymatrix{
 &(X')^e \ar[r]\ar[d]^{F_{X'}^e}  &X^e\ar[d]^{F_X^e} \\
 &X' \ar[r]^<<<<<{\sigma}      &X
}.$$
Then for $\Delta \geq 0$, since $E$ is $\sigma$-exceptional and $X$ is normal, the following commutative diagram of trace maps makes sense
$$\xymatrix{
&\sigma_*F^e_*\mathcal{O}_{X'}(K_{X'} + \ulcorner p^e\sigma^*(D-\Delta) + p^eE\urcorner) \ar[r]^<<<<<{Tr_{\sigma}}\ar[d]^{Tr_{X'}^e} &F^e_*\mathcal{O}_{X}(K_X + \ulcorner p^e(D-\Delta)\urcorner)\ar[d]^{Tr_{X}^e}\\
&\sigma_*\mathcal{O}_{X'}(K_{X'} + \ulcorner \sigma^*D + E\urcorner) \ar[r]^{Tr_{\sigma}} &\mathcal{O}_{X}(K_{X} + \ulcorner D\urcorner)
}.$$
We get a natural map  $Tr_{\sigma}: S_{\sigma^*\Delta}^0(X', K_{X'} + \sigma^*D +E) \to S_{\Delta}^0(X, K_X + D)$ which is injective since $\sigma$ is birational.


For each $\Delta \in \Theta_{D}^{\mathrm{amp}}$, we take an effective divisor $\Delta''$ on $X'$ such that $\sigma^*\Delta + \Delta'' \in \Theta_{\sigma^*D}^{\mathrm{amp}}$. So we can obtain the required map of (iv-1) as follows
\begin{align*}
Tr_{\sigma}: S_{-}^0(X', K_{X'} + \sigma^*D) &  \subseteq  \bigcap_{\Delta \in \Theta_{D}^{\mathrm{amp}}}~(\bigcup_{t\in \mathbb{Q}^+}S_{t(\sigma^*\Delta + \Delta'')}^0(X', K_{X'} + \sigma^*D))\\
 &\subseteq \bigcap_{\Delta \in \Theta_{D}^{\mathrm{amp}}}~(\bigcup_{t\in \mathbb{Q}^+}S_{t\sigma^*\Delta}^0(X', K_{X'} + \sigma^*D)) \\
 &\to \bigcap_{\Delta \in \Theta_{D}^{\mathrm{amp}}}~(\bigcup_{t\in \mathbb{Q}^+}S_{t\Delta}^0(X, K_{X} + D)) = S_{-}^0(X, K_X + D).
\end{align*}

Let us assume $X$ is $\mathbb{Q}$-factorial and prove (iv-2). We have known $S_{-}^0(X', K_{X'} + \sigma^*D)\otimes s_{E} \subseteq S_{-}^0(X', K_{X'} + \sigma^*D + E)$. We still need to prove the inverse inclusion. For this we take $\Delta' \in \mathcal{B}_{\sigma^*D +E}^{\mathrm{nef}}$ and let $\Delta = \sigma_*\Delta'$. As $X$ is assumed $\mathbb{Q}$-factorial, $D-\Delta$ is also a $\mathbb{Q}$-Cartier, nef and big divisor. Notice that $(\sigma^*D +E- \Delta')-\sigma^*(D-\Delta)= \sigma^*\Delta  +E- \Delta'$ is supported in the exceptional locus and is $\sigma$-nef. By negativity lemma, we have $\Delta' -E \geq \sigma^*\Delta \geq 0$. It follows that
$$S_{-}^0(X', K_{X'} + \sigma^*D + E - \Delta') \hookrightarrow S_{-}^0(X', K_{X'} + \sigma^*D -\sigma^*\Delta) \hookrightarrow S_{-}^0(X', K_{X'} + \sigma^*D).$$
Then the desired inclusion follows by the definition (\ref{def:gS}).
\end{proof}

\begin{rem}
Let $\sigma: X' \to X$ be a birational morphism of projective normal varieties of general type. If $X$ is $\mathbb{Q}$-factorial, then by Proposition \ref{prop:F-stable-section} (iv) we have a natural inclusion $S_{-}^0(X', K_{X'} + nK_{X'}) \subseteq S_{-}^0(X, K_{X} + nK_X)$, so the birationality of $S_{-}^0(X, K_{X} + nK_X)$ is implied by that of $S_{-}^0(X', K_{X'} + nK_{X'})$. This fact will be frequently used in the sequel.
\end{rem}

\begin{rem}
We avoid involving $F$-singularities too much for simplicity, at the cost of results of Proposition \ref{prop:F-stable-section} being not in full generality.
We refer the interested reader to \cite[Sec. 2]{Sch14} for techniques to treat singularities.
\end{rem}




\subsection{The behavior of varieties under field extension}\label{sec:bc}
It is known that a variety defined over a non-algebraically closed field may become more singular after inseparable base field extension. Here we collect some results needed in this paper and refer the reader to \cite[Chap. 3.2.2]{Liu02} for a systematical study of this issue. Let $X$ be an integral projective variety over $K$ such that $H^0(X, \mathcal{O}_X) = K$. Then
\begin{itemize}
\item
The variety $X_{\bar{K}}$ is geometrically irreducible, hence $X_{\bar{K}}^{\mathrm{red}}$ is integral; and if $X$ is separable over $\mathrm{Spec}~K$, namely $K(X)/K$ is a separable extension, then $X_{\bar{K}}$ is reduced, hence is integral (\cite[Cor. 2.14 (d) Chap. 3]{Liu02}).
\item
If $K$ is an extension over an algebraically closed field $k$ with the transcendental degree $tr. \deg (K) = 1$, then $X$ is separable over $\mathrm{Spec}~K$ by \cite[Lemma 7.2]{Ba01}.
\item
If $W_{K^{\frac{1}{p^{\infty}}}} \to X_{K^{\frac{1}{p^{\infty}}}}^{\mathrm{red}}$ is a desingularization, then $W_{K^{\frac{1}{p^{\infty}}}}$ is smooth over $K^{\frac{1}{p^{\infty}}}$. Therefore, there exists a finite purely inseparable extension $K'/K$ such that the desingularization $W_{K'} \to X_{K'}^{\mathrm{red}}$ is smooth over $K'$.
\end{itemize}

\begin{prop}\label{prop:ex-over-surface}
Let $X$ be a projective regular surface over a field $K$ such that $H^0(X, \mathcal{O}_X) = K$. If $\sigma: W_{\bar{K}} \to X_{\bar{K}}^{\mathrm{red}}$ is a desingularization\footnote{Lipman \cite{Lip78} proved that an excellent algebraic surface always has a desingularization in the strong sense, namely, the resolution is isomorphic over the regular locus.} then each irreducible component of the exceptional locus of $\sigma$
is a rational curve.
\end{prop}
\begin{proof}
We first take a finite purely inseparable extension $K'/K$ such that the desingularization $\sigma': Y_{K'} \to X_{K'}^{\mathrm{red}}$ is smooth over $K'$. In fact we only need to prove that the exceptional locus of $\sigma'$ geometrically is a union of rational curves.
For sufficiently large $e$, the function field $K(X)$ is between $K(Y_{K'})^{p^e}$ and $K(Y_{K'})$, this gives a rational map $X \dashrightarrow Y_{K'}^{p^e}$. Let $Y$ be the normalization of $Y_{K'}^{p^e}$ in $K(X)$. Then we have a natural commutative diagram
$$\xymatrix{
&Y_{K'} \ar[r]\ar[d] &Y \ar[d]\\
&X_{K'}^{\mathrm{red}} \ar[r]  &X
}$$
such that $Y$ is normal, $Y \to X$ is a projective birational morphism and $Y_{K'} \to Y$ is finite and purely inseparable, hence is an homeomorphism. Let $\tilde{Y} \to Y$ be a regular resolution. Since $X$ is regular, a relative minimal model program starting with $\tilde{Y}$ over $X$ must end up with $X$, we see that the exceptional locus is a union of curves of arithmetic genus zero (\cite{Sh66}). This is enough to imply the statement.
\end{proof}

\subsection{The behaviour of the canonical bundle under purely inseparable base changes}\label{sec:canonical-bundle-bc}
First recall an important observation due to Tanaka \cite{Tan18}.
Let $K$ be a field of characteristic $p$ and $\bar{K}$ the algebraic closure of $K$. Let $X$ be a normal projective variety over $K$ such that $H^0(X, \mathcal{O}_X) = K$. Assume $X_{\bar{K}}$ is not reduced. By \cite[Lemma 2.3 and 2.4]{Tan18} we have
the following commutative diagram
$$\xymatrix{
&X^{(1)}\ar[rd]_{\sigma'}\ar[d]\ar[rrd]^{\sigma_1} &  & \\
&X_{K_1}\ar[r]\ar[d]  &X_{K_1'} \ar[r]^{\pi}\ar[d]  &X\ar[d] \\
&\mathrm{Spec}~K_1 \ar[r]                 &\mathrm{Spec}~K_1' \ar[r]     &\mathrm{Spec}~K
}$$
where
\begin{itemize}
\item
$K_1'/K$ is a purely inseparable extension of degree $p$ such that $X_{K_1'}$ is integral but not normal;
\item
$\sigma': X^{(1)} \to X_{K_1'}$ is the normalization map, and $K_1= H^0(X^{(1)}, \mathcal{O}_{X^{(1)}})$ is an inseparable extension over $K_1'$.
\end{itemize}
Note that $X_{K_1}$ is not reduced and the natural map $X^{(1)} \to X_{K_1}^{\mathrm{red}}$ coincides with the normalization.
Comparing the dualizing sheaves, we can write that
$$K_{X^{(1)}} \sim \sigma'^*K_{X_{K_1'}} - C \sim \sigma'^*\pi^*K_{X_{K_1'}} - C \sim \sigma_1^*K_X - C$$
with the following remarks
\begin{itemize}
\item
$X_{K_1'}$ is Gorenstein in codimension one ($G_1$) and satisfies Serre condition $S_2$, as its dualizing sheaf coincides with $\pi^*\omega_X$ and hence is reflexive, we may write it into the divisorial form $\mathcal{O}_{X_{K_1'}}(K_{X_{K_1'}})$ and also call $K_{X_{K_1'}}$ the canonical divisor\footnote{On an $S_2$ variety $X$, a rank one reflexive coherent sheaf $\mathcal{L}$ is invertible in codimension one and is isomorphic to $i_{U*}(\mathcal{L}|_U)$ where $U$ is the maximal open subset of $X$ such that $\mathcal{L}|_U$ is invertible.};
\item $C > 0$ arises from the conductor of the normalization (\cite[2.6]{Re94}).
\end{itemize}
We can repeat this process and finally get a finite sequence
$$\xymatrix{&X=X^{(0)}/K &X^{(1)}/K_1 \ar[l] &\cdots\ar[l] &X^{(n-1)}/K_{n-1} \ar[l] &X^{(n)}/K_n \ar[l]}$$
where $X^{(n)}$ is geometrically reduced. Then we can do a further finite purely inseparable field extension $L/K_n$ such that the normalization $Z_L \to (X_n)_L$ is geometrically normal over $L$. In this way, we obtain a normalization $Z_{\bar{K}} \to X_{\bar{K}}^{\mathrm{red}}$. If denoting by $\sigma: Z_{\bar{K}} \to X$ the natural morphism, there exists an effective Weil divisor $E$ on $Z_{\bar{K}}$, which arises from the conductor in the process of doing the normalization, such that
$$K_{Z_{\bar{K}}} \sim \sigma^*K_X -E.$$
Moreover by \cite[Theorem 3.16]{Tan19}), we may write that $E= (p-1)E'$ for some effective Weil divisor $E'$ on $Z_{\bar{K}}$.

\subsection{Frobenius stable sections on generic fibers and geometric fibers}\label{sec:Frob-stable-section-bc}
Though we originally aim to treat only normal varieties, non-normal varieties appear as intermediates when doing  inseparable base changes, we package the treatment in the following theorem.

\begin{thm}\label{thm:bir-geo-gen}
Let $K$ be an $F$-finite field and $\bar{K}$ the algebraic closure of $K$. Let $X$ be a normal projective variety over $K$ such that $H^0(X, \mathcal{O}_X) = K$. Let $\sigma: Z_{\bar{K}} \to X_{\bar{K}}^{\mathrm{red}}$ be the normalization map. Let $D$ be a nef and big $\mathbb{Q}$-Cartier $\mathbb{Q}$-divisor on $X$, and denote by $\bar{D}$ its pullback via the natural map $Z_{\bar{K}} \to X$. If $S_{*}^0(Z_{\bar{K}}, K_{Z_{\bar{K}}} + \bar{D})$ is birational (resp. nonzero), then so is $S_{*}^0(X, K_X + D)$, where ``$*$'' can be ``$-$'' or an effective divisor $\Delta$ on $X$ (automatically
$*= \bar{\Delta}(:=\Delta|_{Z_{\bar{K}}})$ on $Z_{\bar{K}}$).
\end{thm}
\begin{proof}
In the following we only consider the case ``$*$''$=$``$-$'' and divide the proof into four steps. If ``$*$'' is taken as a fixed effective divisor we only need Step 1, 3 and 4.
\smallskip

\textbf{Step 1:} Let $K'/K$ be a field extension where $K'$ is also $F$-finite and let $\sigma': X' \to X_{K'}^{\mathrm{red}}$ denote the normalization morphism. For an effective $\mathbb{Q}$-divisor $\Delta$ on $X$, the trace map of $\sigma'$ induces a natural $K'$-linear map $\eta_{K', \Delta}: S_{\Delta'}^0(X', K_{X'} + D') \to S_{\Delta}^0(X, K_X + D)\otimes_K K'$ where $D' = D|_{X'}, \Delta'=\Delta|_{X'}$.

Proof of Step 1.
We consider the following commutative diagram
$$\xymatrix{
&(X')^e\ar[r]^>>>>{\sigma'_e}\ar[d]_{F^e_{X'}} &X^e_{K'}=X^e\otimes_K K' \ar[r]^<<<<<{\pi_e}\ar[d]^{\eta_e}  &X^e\ar[d]^{F_X^e} \\
& X'\ar[r]^>>>>>>{\sigma'}                                        &X_{K'}=X\otimes_K K' \ar[r]^<<<<<{\pi}      &X
}.$$
Here we remark that both $X_{K'}$ and $X^e_{K'}$ are $G_1$ and $S_2$, and the dualizing sheaves $\mathcal{O}_{X_{K'}}(K_{X_{K'}})$ and $\mathcal{O}_{X^e_{K'}}(K_{X^e_{K'}})$ coincide with the pullback of the dualizing sheaves of $X$ and $X^e$ respectively. First, since $\pi$ is a flat base change, we have the trace map
{\small \begin{align*}
& \eta_{e*}[\mathcal{O}_{X^e_{K'}}(K_{X^e_{K'}} + \pi^{*}_e\ulcorner p^e(D-\Delta)\urcorner)(: =\pi^{*}_e\mathcal{O}_{X^e}(K_{X^e} + \ulcorner p^e(D-\Delta)\urcorner))] \\
&\cong \pi^*F_{X*}^e\mathcal{O}_{X^e}(K_{X^e} + \ulcorner p^e(D-\Delta)\urcorner) \xrightarrow{\pi^*Tr_X^e} \mathcal{O}_{X_{K'}}(K_{X_{K'}} + \pi^*\ulcorner(D-\Delta)\urcorner)(:= \pi^*\mathcal{O}_{X}(K_{X} + \ulcorner (D-\Delta)\urcorner)).
\end{align*}}
For $e\gg0$ the image of the global sections of the above map coincides with $S^0(X, K_X + (D-\Delta))\otimes_K K'$. Second, by $\sigma'^{*}_e\pi^{*}_e\ulcorner p^e(D-\Delta)\urcorner \geq \ulcorner p^e(D'-\Delta')\urcorner$, the trace map
$$
Tr_{\sigma'_e}:~ \sigma'_{e*}(\mathcal{O}_{X'}(K_{X'} + \ulcorner p^e(D'-\Delta')\urcorner))
\to \mathcal{O}_{X^e_{K'}}(K_{X^e_{K'}} + \pi^{*}_e\ulcorner p^e(D-\Delta)\urcorner)
$$
makes sense in codimension one, and it in fact can extend to the whole variety because $\mathcal{O}_{X^e_{K'}}(K_{X^e_{K'}} + \pi^{*}_e\ulcorner p^e(D-\Delta)\urcorner)$ is reflexive. In summary, we obtain the following commutative diagram
$$\xymatrix{
&\eta_{e*}\sigma'_{e*}\mathcal{O}_{X'}(K_{X'} + \ulcorner\ p^e(D'-\Delta')\urcorner) \ar[r]^{Tr_{\sigma'_e}}\ar[d]_{Tr^e_{X'}} &\eta_{e*}\mathcal{O}_{X^e_{K'}}(K_{X^e_{K'}} + \pi^{e*}\ulcorner\ p^e(D-\Delta)\urcorner) \ar[d]_{\pi^*Tr^e_{X}}  \\
&\sigma'_* \mathcal{O}_{X'}(K_{X'} + \ulcorner(D'-\Delta')\urcorner)\ar[r]^{Tr_{\sigma'}}                 &\pi^*\mathcal{O}_{X}(K_{X} + \ulcorner (D-\Delta)\urcorner)~.
}$$
For $e \gg 0$, the map between the images of the global sections of the vertical trace maps is nothing but the desired map $\eta_{K', \Delta}: S_{\Delta'}^0(X', K_{X'} + D') \to S_{\Delta}^0(X, K_X + D)\otimes_K K'$.
\smallskip

\textbf{Step 2:} The trace map $Tr_{\sigma}$ induces the $\bar{K}$-linear map $\eta: S_{-}^0(Z_{\bar{K}}, K_{Z_{\bar{K}}} + D) \to  S_{-}^0(X, K_{X} + D)\otimes_K \bar{K}$.

Proof of Step 2.
Note that $\sigma^*\Theta_{D}^{\mathrm{amp}} \subseteq \Theta_{\sigma^*D}^{\mathrm{amp}}$. Then the desired map is obtained as follows
\begin{align*}
\eta: S_{-}^0(Z_{\bar{K}}, K_{Z_{\bar{K}}} + \bar{D}) &\subseteq \bigcap_{\Delta \in \Theta_{D}^{\mathrm{amp}}}~(\bigcup_{t\in \mathbb{Q}^+}S_{t\sigma^*\Delta}^0(Z_{\bar{K}}, K_{Z_{\bar{K}}} + \bar{D}))\\
 & \xrightarrow{\bigcap_{\Delta \in \Theta_{D}^{\mathrm{amp}}}~(\bigcup_{t\in \mathbb{Q}^+}\eta_{\bar{K}, t\Delta})} \bigcap_{\Delta \in \Theta_{D}^{\mathrm{amp}}}~(\bigcup_{t\in \mathbb{Q}^+}S_{t\Delta}^0(X, K_X + D))\otimes_K\bar{K} \\
 & = S_{-}^0(X, K_{X} + D)\otimes_K \bar{K}
\end{align*}
where the map appearing in the second row is from Step 1.

\smallskip

\textbf{Step 3:} We assume that $X$ is geometrically reduced and prove the theorem.

Proof of Step 3.
Since $X_{\bar{K}}$ is reduced, thus $Z_{\bar{K}} \to X_{\bar{K}}$ is a birational map. By Step 2 we have a natural inclusion
$$S_{-}^0(Z_{\bar{K}}, K_{Z_{\bar{K}}} + D) \hookrightarrow S_{-}^0(X, K_{X} + D)\otimes_K \bar{K}.$$
Since $S_{-}^0(Z_{\bar{K}}, K_{Z_{\bar{K}}} + D)$ induces a birational map of $Z_{\bar{K}}$ (resp. is nonzero), $S_{-}^0(X, K_{X} + D)\otimes_K \bar{K}$ induces a birational map of $X_{\bar{K}}$ (resp. is nonzero) too. Then we conclude that $S_{-}^0(X, K_X + D)$ is birational (resp. is nonzero).
\smallskip

\textbf{Step 4:} We assume that $X$ is not geometrically reduced and prove the theorem.

Proof of Step 4.
Recall the following commutative diagram from Section \ref{sec:canonical-bundle-bc}
$$\xymatrix{
&X^{(1)}\ar[rd]\ar[d] &  & \\
&X_{K_1}\ar[r]\ar[d]  &X_{K_1'} \ar[r]\ar[d]  &X\ar[d] \\
&\mathrm{Spec}~K_1 \ar[r]                 &\mathrm{Spec}~K_1' \ar[r]     &\mathrm{Spec}~K.
}$$
As in the previous case the trace map induces an inclusion of $K_1'$-linear spaces
$$S_{-}^0(X^{(1)}, K_{X^{(1)}} + D)\hookrightarrow S_{-}^0(X, K_{X} + D)\otimes_K K'_1.$$
If the  $K_1$-linear space $S_{-}^0(X^{(1)}, K_{X^{(1)}} + D)$ defines a birational map of $X^{(1)}$ (resp. is nonzero), then it also defines a birational map of $X_{K_1'}$ as $K_1'$-linear space, in turn we can prove that $S_{-}^0(X, K_{X} + D)$ is birational (resp. is nonzero).

If $X^{(1)}$ is already geometrically reduced, then the normalization of $(X^{(1)}_{\bar{K}})^{\mathrm{red}}$ coincides with $Z_{\bar{K}}$, in turn we conclude the theorem by combining the result of Step 3 with the assertion proved in the previous paragraph.
Otherwise, we can repeat this process and finally get a chain of purely inseparable maps
$$\xymatrix{&X=X^{(0)}/K &X^{(1)}/K_1 \ar[l] &\cdots\ar[l] &X^{(n)}/K_n \ar[l]}$$
such that $X^{(n)}$ is geometrically reduced. We can complete the proof by doing induction.
\end{proof}

To compare the behaviours of Frobenius stable adjoint linear systems on the general fibers and the generic fiber, we have the following result.
\begin{cor}\label{cor:bir-general-fiber}
Let $f: X \to Y$ be a fibration of normal varieties over an algebraically closed, uncountable field $k$ of characteristic $p$. Let $D$ be a nef and big $\mathbb{Q}$-Cartier $\mathbb{Q}$-divisor on $X$. Let $F$ be a general fiber and denote by $G \to F^{\mathrm{red}}$ the normalization.
Assume that $S_{-}^0(G, K_G+ D|_G)$ is birational (resp. non-zero). Then $S_{-}^0(X_{\eta}, K_{X_{\eta}}+ D|_{X_{\eta}})$ is birational (resp. non-zero).
\end{cor}
\begin{proof}
By Theorem \ref{thm:bir-geo-gen} we only need to prove the analogous assertions for the geometric generic fiber. To relate the general fiber and the geometric generic fiber, we shall use the trace maps of the relative Frobenius maps.

We may assume $f$ is flat and $Y$ is regular. There exists a quasi-finite, flat, purely inseparable base change $Y' \to Y$ such that, if $Z$ denotes the normalization of $(X':=X\times_Y Y')^{\mathrm{red}}$, the geometric generic fiber $Z_{\bar{\eta}}$ of the fibration $g: Z \to Y'$ is normal. We can shrink $Y'$ to assume $g$ is flat and $Y'$ is regular, and if $F = X_y$ for some closed point $y \in Y$ and $y' \in Y'$ is the point over $y$ then $G \cong Z_{y'}$.  By Theorem \ref{thm:bir-geo-gen} we only need to prove the assertion for the firation $g: Z \to Y'$. We consider the following commutative diagram
$$\xymatrix@C=2cm{&Z=Z^e \ar[r]^{F_{Z/Y'}^e}\ar[d]^{g^e}\ar@/^2pc/[rr]|{F_Z^e} &Z_{Y'^e}=Z\times_{Y'} Y'^e \ar[r]^{\pi^e}\ar[d]^{g_e'}  & Z\ar[d]^g \\
&Y'\ar@{=}[r]   &Y'^e\ar[r]^{F_{Y'}^e}  &Y'
}.$$
Fix an effective divisor $\Delta$ on $Z$ and consider the trace map of the relative Frobenius map $F_{Z/Y'}^e$
{\small \begin{align*}
Tr_{Z/Y'}^e:~ & g^e_*\mathcal{O}_{Z^e}(K_{Z^e/Y'^e} + \ulcorner p^e(D-\Delta)\urcorner|_{Z^e}) \\
&\to g'_{e*}(\pi^{e*}\mathcal{O}_Z(K_{Z/Y} + \ulcorner D\urcorner|_Z)) \cong F_{Y'}^{e*}g_*\mathcal{O}_Z(K_{Z/Y'} + \ulcorner D\urcorner|_Z)
\end{align*}}
where the isomorphism holds because $F_{Y'}^e$ is flat. We denote  the image of the above trace map by $S_{\Delta}^eg_*\mathcal{O}_Z(K_{Z/Y'} + D) \subseteq F_Y^{e*}g_*\mathcal{O}_Z(K_{Z/Y'} + \ulcorner D-\Delta\urcorner)$.

Applying Propostion \ref{prop:F-stable-section} (ii), we can find a reduced subvariety $\bar{T}$ on $Z_{\bar{\eta}}$ such that $\mathrm{Supp}~D_{\bar{\eta}} \subseteq \bar{T}$ and for any $\bar{\Delta} \in \Theta_{D_{\bar{\eta}}}^{\mathrm{amp}}$, if $\mathrm{Supp}~\bar{\Delta} = \bar{T}$ then
$$S_{\bar{\Delta}}^0(Z_{\bar{\eta}}, (K_{Z/Y'} + D)|_{Z_{\bar{\eta}}}) \subseteq S_{-}^0(Z_{\bar{\eta}}, (K_{Z/Y'} + D)|_{Z_{\bar{\eta}}}).$$
We can do some further quasi-finite flat base changes to assume there exists a subvariety $T \subset Z$ such that $T_{\bar{\eta}} = \bar{T}$, and that each irreducible component of $T$ is flat and separable over $Y'$, hence its restriction on $Z_{\bar{\eta}}$ is still reduced. As a consequence, if $B$ is a divisor supported on $T$ then
$$\ulcorner B\urcorner|_{Z_{\bar{\eta}}}= \ulcorner B|_{Z_{\bar{\eta}}}\urcorner~ \mathrm{and}~\ulcorner B\urcorner|_{G} = \ulcorner B|_{G}\urcorner.$$

From now on we fix an effective $\mathbb{Q}$-Cartier  $\mathbb{Q}$-divisor $\Delta$ on $Z$ such that $\mathrm{Supp}~\Delta=T$ and that $D|_Z - \Delta$ is ample over $Y'$. By shrinking $Y'$ again, we may assume  $\mathrm{Supp}~D \subseteq T$ and for any geometric point $\zeta \in Y'$ and natural number $e$
$$\ulcorner p^e(D-\Delta) \urcorner|_{Z_{\zeta}} = \ulcorner p^e(D-\Delta)|_{Z_{\zeta}} \urcorner,$$
and in turn we obtain a natural $k(\zeta)$-linear map
$$S_{\Delta}^eg_*\mathcal{O}_Z(K_{Z/Y'} + D) \otimes_{\mathcal{O}_{Y'}^{\frac{1}{p^e}}} k(\zeta)^{\frac{1}{p^e}} \to S_{\Delta}^e(Z_{\zeta}, (K_{Z/Y'} + D)|_{Z_{\zeta}}).$$
As $k(\zeta)$ is algebraically closed, we may omit the Frobenius twist if no confusion occurs.
Let $t_n= \frac{1}{n}$. By replacing $\Delta$ with a small multiple, we may assume that for all $n$
$$S_{-}^0(Z_{\bar{\eta}}, (K_{Z/Y'} + D)|_{Z_{\bar{\eta}}}) = S_{t_n\Delta}^0(Z_{\bar{\eta}}, (K_{Z/Y'} + D)|_{Z_{\bar{\eta}}}).$$
Then for each $n$ there exists a positive integer $e_n$ such that for any $e\geq e_n$,
{\small \begin{align}\label{eq:fs-geo-gen}S_{t_n\Delta}^eg_*\mathcal{O}_Z(K_{Z/Y'} + D) \otimes k(\bar{\eta}) \cong S_{t_n\Delta}^e(Z_{\bar{\eta}}, (K_{Z/Y'} + D)|_{Z_{\bar{\eta}}}) = S_{-}^0(Z_{\bar{\eta}}, (K_{Z/Y'} + D)|_{Z_{\bar{\eta}}}).\end{align}}
For each $n$ and $e \geq e_n$, there exists a nonempty open subset $U_{n,e} \subseteq Y'$ such that $S_{t_n\Delta}^eg_*\mathcal{O}_Z(K_{Z/Y'} + D)|_{U_{n,e}}$ is locally free of rank $\dim_{k(\bar{\eta})} S_{-}^0(Z_{\bar{\eta}}, (K_{Z/Y'} + D)|_{Z_{\bar{\eta}}}) $ and for any closed point $y' \in U_{n,e}$
{\small\begin{align}\label{eq:fs-gen}S_{t_n\Delta}^eg_*\mathcal{O}_Z(K_{Z/Y'} + D)\otimes k(y') \cong S_{t_n\Delta}^e(Z_{y'}, (K_{Z/Y'} + D)|_{Z_{y'}}).\end{align}}
Therefore, for each $y' \in \Xi = \bigcap_n \bigcap_e U_{n,e}$, $\dim_{k(y')}S_{t_n\Delta}^e(Z_{y'}, (K_{Z/Y'} + D)|_{Z_{y'}})$ is independent of $n$ and sufficiently large $e$, hence
$$S_{t_n\Delta}^e(Z_{y'}, (K_{Z/Y'} + D)|_{Z_{y'}}) \cong S_{t_n\Delta}^0(Z_{y'}, (K_{Z/Y'} + D)|_{Z_{y'}}) \supseteqq S_{-}(Z_{y'}, (K_{Z/Y'} + D)|_{Z_{y'}}).$$
Since $k$ is uncountable, the locus $\Xi$ is dense in $Y'$.

If for general $y' \in Y'$, $S_{-}(Z_{y'}, (K_{Z/Y'} + D)|_{Z_{y'}})$ is birational, then there is a dense subset $\Xi' \subseteq \Xi$ such that for any $y' \in \Xi'$ and $e\in \mathbb{N}$, $S_{t_n\Delta}^e(Z_{y'}, (K_{Z/Y'} + D)|_{Z_{y'}})$ is birational. From this we conclude by (\ref{eq:fs-gen}) that for the relative map
{\small $$Z/Y' \dashrightarrow \mathrm{Proj}_{\mathcal{O}_{Y'}}(\bigoplus_l \mathrm{Sym}^l(S_{\Delta}^eg_*\mathcal{O}_Z(K_{Z/Y'} + D))$$}
is a birational map, which, by (\ref{eq:fs-geo-gen}), implies that $S_{-}^0(Z_{\bar{\eta}}, (K_{Z/Y'} + D)|_{Z_{\bar{\eta}}})$ is birational.

The nonvanishing assertion follows similarly.
\end{proof}



\subsection{Birational Criterion}
The following birational criterion is well known to experts, but here we provide a detailed proof for the conveniences of the reader.
\begin{thm}\label{thm:bir}
Let $f: X \to Y$ be a dominant morphisms of integral varieties over an algebraically closed field $k$, with integral generic fiber $X_{\eta}$. Let $D$ be a Weil divisor on $X$, and let $V \subseteq H^0(X, D)$ be a finite dimensional linear subspace, which is also regarded as a subspace of $H^0(Y, f_*\mathcal{O}_X(D))$. Let $\mathcal{V} \subseteq f_*\mathcal{O}_X(D)$ be the subsheaf generated by $V$. Assume that

(i) $|V||_{X_{\eta}}$ defines a birational map of $X_{\eta}$, \\
and one of the following conditions holds

(ii) there is an open dense subset $Y' \subseteq Y$ such that for any $y \in Y'$, $\mathcal{V}\otimes \mathcal{I}_y$ is globally generated on $Y'$ by
$V_y:=V \bigcap H^0(Y, \mathcal{V}\otimes \mathcal{I}_y)$;

(ii') there exists a linear system $|H|$ on $Y$ and an effective Weil divisor $E$ on $X$ such that $|H|$ defines a birational map of $Y$ and $f^*|H|+ E \subseteq |V|$.\\
Then $|V|$ defines a birational map of $X$.
\end{thm}
\begin{proof}

We first show the theorem under the assumptions (i) and (ii).
We may assume $Y' = \mathrm{Spec} A$ is affine and restrict ourselves on a nonempty open affine subset $X' = \mathrm{Spec} B$ with $f(X') \subseteq Y'$. By shrinking $X'$ and choosing a local generator of $\mathcal{O}_X(D)|_{X'}$, we may assume $V|_{X'} = \mathrm{Span}_k\{1, g_1, \cdots, g_r\}$ where $g_1, \cdots, g_r \in B$. The condition (i) means $K(A)(g_1, \cdots, g_r) = K(B)$. By shrinking $X'$ again we may assume
\begin{itemize}
\item
$B= A[g_1, \cdots, g_r]_g$ for some $g \in A[g_1, \cdots, g_r]$.
\end{itemize}
It follows that for $x \in X'$ and $y=f(x) \in Y'$,
\begin{itemize}
\item[(a)]
$X'_y \cong \mathrm{Spec}~A[g_1, \cdots, g_r]_g\otimes_Ak(y) \cong \mathrm{Spec}~k[\bar{g}_1, \cdots, \bar{g}_r]_{\bar{g}}$, thus the ideal sheaf $\mathcal{I}_{X'_y,x}$ is globally generated by $\bar{V} (:= \mathrm{Span}_k\{\bar{g}_1, \cdots, \bar{g}_r\}) \cap \frac{I_x}{I_y\cdot B}$.
\end{itemize}
Define the $A$-module $\mathcal{W} = \mathrm{Span}_A\{1, g_1, \cdots, g_r\} \subset B$, which generate a coherent sheaf $\widetilde{\mathcal{W}}$ on $Y'$ such that $\mathcal{V}|_{Y'} = \widetilde{\mathcal{W}}$. Take a closed point $y \in \mathrm{Spec} A$, denote by $I_y \subset A$ the ideal of $y$ and $\mathcal{W}_y = \mathcal{W} \cap I_y \cdot B$. Then
\begin{itemize}
\item[(b)]
we may identify $V_y = V \cap I_y\cdot B$, and by the condition (ii) we have $\mathcal{W}_y = \mathrm{Span}_AV_y$, and as a consequence the ideal $I_{y}\cdot B$ is generated by $V_y$.
\end{itemize}
Then considering the following exact sequence
$$0 \to \mathcal{I}_y\cdot \mathcal{O}_{X',x} \to \mathcal{I}_{X',x}  \to \mathcal{I}_{X'_y,x}\to 0$$
by (a) and (b) we can conclude that $\mathcal{I}_{X',x}$ is globally generated by $V \cap I_x$. The above arguments show that $|V||_{X'}$ defines an isomorphism from $X'$ to its image, hence $|V|$ defines a birational map of $X$.


\smallskip

Now assume (i) and (ii'). To apply the above argument, we require moreover that $X' \cap E=\emptyset$. Then by the assumption (ii') we can find an open subset $Y' \subset Y$ such that for any $y\in Y'$, the sheaf $\mathcal{I}_y\cdot \mathcal{O}_{X',x}$ is generated by $V_y$. This combining with (a) shows the desired statement.
\end{proof}

\section{Linear systems on varieties equipped with certain fibrations}\label{sec:main-bir-crt}

In this section, we work over an algebraically closed field $k$ of characteristic $p>0$. We will prove the nonvanishing and birational criteria in Theorem \ref{thm:intr-bir-crt1} and \ref{thm:intr-bir-crt2}.


\subsection{} We restate the first criterion as follows. 
\begin{thm}\label{thm:bir-criterion-for-induction}
Let $f:X \to Y$ be a fibration of normal projective varieties and let $d = \dim Y$.
Let $D$ be a nef and big $\mathbb{Q}$-Cartier $\mathbb{Q}$-divisor on $X$, and $H, \tilde{H}$ two $\mathbb{Q}$-Cartier Weil divisors on $Y$ such that $|H|$ defines a generically finite map and $|\tilde{H}|$ is birational.

(i) If $S_{-}^0(X_{\eta}, K_{X_{\eta}} + D|_{X_{\eta}}) \neq 0$ then $S_{-}^0(X, K_X +D + f^*sH)\neq 0$ for any $s \geq d$.

(ii) If $S_{-}^0(X_{\eta}, K_{X_{\eta}} + D|_{X_{\eta}})$ is birational then $S_{-}^0(X, K_X +D + f^*sH)$ is birational for $s \geq d+1$; and if moreover $S_{-}^0(X, K_{X} +D + df^*H - f^*\tilde{H}) \neq 0$ then $S_{-}^0(X, K_X +D + f^*dH)$ is birational.
\end{thm}
\begin{proof}
To ease the situation, we first blow up $Y$ to make the movable part of $|H|$ have no base point, then blow up $X$ along some locus disjoint to $X_{\eta}$ to make sure the rational map $X\to Y$ is still a morphism, and replace $H$ with the movable part and replace $D, \tilde{H}$ with their pull-backs. Remark that under the above birational modifications the assumptions in the theorem still hold. From now on, we may assume $|H|$ is base point free.

Next we claim that for each $m \in \mathbb{N}$ there exists an effective divisor $\Delta_m$ on $X$ such that $D-\Delta_m$ is ample, and that $S_{-}^0(X, K_X +D + f^*mH) = S_{\Delta_m}^0(X, K_X +D + f^*mH)$. To prove this claim, we fix an effective divisor $\Delta \in \Theta_D^{\mathrm{amp}}$. By Proposition \ref{prop:F-stable-section},
there exists a divisor $\Delta'_m \in \Theta_{D + f^*mH}^{\mathrm{amp}}$ such that
$S_{-}^0(X, K_X +D + f^*mH) = S_{t\Delta'_m}^0(X, K_X +D + f^*mH)$ for any sufficiently small positive rational number $t$. We may take a sufficiently small $s\in \mathbb{Q}^{+}$ such that $D-\Delta - s\Delta'_m$ is ample. Set $\tilde{\Delta}_m = \Delta + s\Delta'_m$. Then for any $t\in \mathbb{Q}^{+}$,
$$S_{t\tilde{\Delta}_m}^0(X, K_X +D + f^*mH) \subseteq S_{ts\Delta'_m}^0(X, K_X +D + f^*mH) \subseteq S_{-}^0(X, K_X +D + f^*mH).$$
By the definition (\ref{def:S0}) we can take a sufficiently small $t_0\in \mathbb{Q}^{+}$ such that  both the above inclusions attain equality, then we may let $\Delta_m=t_0(\Delta + s\Delta'_m)$.

We fix a sufficiently small rational number $t \ll1$ and may assume that
$$S_{-}^0(X_{\eta}, K_{X_{\eta}} + D|_{X_{\eta}}) \subseteq S_{t\Delta_m|_{X_{\eta}}}^0(X_{\eta}, K_{X_{\eta}} + D|_{X_{\eta}}).$$
By the construction of $\Delta_m$, we only need to prove the following Theorem \ref{thm:bir-criterion-1}.
\end{proof}

\begin{thm}\label{thm:bir-criterion-1}
Let the notation and assumptions be as in Theorem \ref{thm:bir-criterion-for-induction}. Assume moreover that $|H|$ is free of base point. Let $\Delta$ be an effective $\mathbb{Q}$-Cartier $\mathbb{Q}$-divisor on $X$ such that $D -\Delta$ is ample.

(i) If $S_{\Delta|_{X_{\eta}}}^0(X_{\eta}, K_{X_{\eta}} + D|_{X_{\eta}}) \neq 0$, then $S_{\Delta}^0(X, K_X +D + f^*sH) \neq 0$ for any $s \geq d$.

(ii) If $S_{\Delta|_{X_{\eta}}}^0(X_{\eta}, K_{X_{\eta}} + D|_{X_{\eta}})$ is birational, then $S_{\Delta}^0(X, K_X +D + f^*sH)$ is birational for any $s \geq d+1$, and if in addition $S_{\Delta}^0(X, K_{X} +D + df^*H - f^*\tilde{H}) \neq 0$ then $S_{\Delta}^0(X, K_X +D + f^*dH)$ is birational.
\end{thm}
\begin{proof}
Let $\mu: Y \to Y'$ be the associated map to $H$. Then $H'= \mu_*H$ is ample, and the linear system $|H'|$ is base point free. Denote by $f': X \to Y'$ the natural morphism. We fit these varieties into the following commutative diagram
$${\small \xymatrix{&X\ar[d]^{f}\ar[rd]^{f'}  &\\
&Y\ar[r]^{\mu}   &Y'}}.$$
\smallskip

Let
$$\mathcal{F}_{s}^e= f'_*(F^{e}_*\mathcal{O}_X(K_X + \ulcorner p^e(D-\Delta)\urcorner)) \otimes \mathcal{O}_{Y'}(sH')$$
and
$$\mathcal{V}_{s}^e = \mathrm{Im}[f'_*Tr^e_{\Delta}: \mathcal{F}_{s}^e  \to f'_*\mathcal{O}_X(K_X + \ulcorner D\urcorner) \otimes \mathcal{O}_{Y'}(sH')]$$
and
$$V_{s}^e = \mathrm{Im}(Tr^e_{\Delta}: H^0(Y', \mathcal{F}_{s}^e) \to H^0(Y', \mathcal{V}_{s}^e)).$$
For sufficiently large $e$ we have natural isomorphisms
$$\mathcal{V}_{s}^e\otimes k(\eta) \cong S_{\Delta|_{X_{\eta}}}^0(X_{\eta}, K_{X_{\eta}} + D|_{X_{\eta}})~\mathrm{and}~ V_{s}^e \cong V_{s}(:=S_{\Delta}^0(X, K_X +D + f^*sH)).$$

\begin{lem}\label{lem:Mumford-reg}
There exists a positive integer $e_0$ such that for any $e>e_0$,

(a)  if $s \geq d$ then the sheaf $\mathcal{V}_{s}^e$ is globally generated by $V_{s}$, and

(b) if $s \geq d+1$ then for any closed point $y' \in Y'$ over which $f'$ is flat, the sheaf $\mathcal{I}_{y'}\cdot\mathcal{V}_{s}^e$ is globally generated by $V_{s,y'}=V_{s}\cap H^0(Y',\mathcal{I}_{y'}\cdot\mathcal{V}_{s}^e)$.
\end{lem}
Granted this lemma, under the assumptions of the theorem, the assertion (i) follows from (a); by applying Theorem \ref{thm:bir}, the first part of the assertion (ii) follows from (b) and the second part follows from (a).
\smallskip

We start to prove Lemma \ref{lem:Mumford-reg}. By the construction, we only need to verify that $\mathcal{F}_{\Delta, s}^e$ (reps. $\mathcal{I}_{y'}\cdot\mathcal{F}_{\Delta, s}^e$) satisfies Mumford regularity if $s\geq d$ (resp. $s \geq d+1$).

Since $D-\Delta$ is an ample $\mathbb{Q}$-Cartier $\mathbb{Q}$-divisor, applying relative Fujita vanishing (Theorem \ref{thm:var-fujita-vanishing}) there exists $e_1$ such that
\begin{itemize}
\item[$\diamondsuit$:]
for any $e\geq e_1$ and $i>0$, $R^if'_*(F^{e}_*\mathcal{O}_X(K_X + \ulcorner p^e(D-\Delta)\urcorner)) =0$.
\end{itemize}
Consider the Lerray spectral sequence associated to $R\Gamma\circ Rf'_*(F^{e}_*\mathcal{O}_X(K_X + \ulcorner p^e(D-\Delta)\urcorner))$. Applying Fujita vanishing we show that there
exists $e_0 \geq e_1$ such that for any $e\geq e_0$, if $s\geq d$ then
$$\clubsuit:~H^i(Y', \mathcal{F}_{s}^e(-jH')) \cong H^i(X, K_X + \ulcorner p^e(D-\Delta)\urcorner + (s-j)p^ef'^*H') = 0 ~\mathrm{for}~j\leq d,~i>0.$$
In particular, the sheaf $\mathcal{F}_{\Delta, s}^e$ satisfies Mumford regularity if $s\geq d$.

Fix $e\geq e_0$. To prove $\mathcal{I}_{y'}\cdot\mathcal{F}_{\Delta, s}^e$ satisfies Mumford regularity, we take $d$ general hypersurfaces $H_1', H_2', \cdots, H_d' \in |H'|$ passing through $y'$.
We may assume that $\mathrm{Supp} (\bigcap_{t=1}^{t=d} H_t')$ consists of finitely many isolated points and is contained in the flat locus of $f'$. Then Kozul complex gives an exact sequence
$$0 \to \mathcal{O}_{Y'}(-\sum_{t=1}^{t=d}H_t') \to \cdots \to \bigoplus_{t=1}^{t=d}\mathcal{O}_{Y'}(-H_t') \to \mathcal{J} \to 0$$
where $\mathcal{J} \subset \mathcal{O}_{Y'}$ is an ideal sheaf such that $\mathrm{Supp}~\mathcal{O}_{Y'}/\mathcal{J} =\mathrm{Supp} ~(\bigcap_{t=1}^{t=d} H_t')$. By the condition ($\diamondsuit$), we know that
the sheaf $\mathcal{F}_{s}^e$ is locally free around $\mathrm{Supp} (\bigcap_{t=1}^{t=d} H_t')$. Tensoring the above exact sequence with $\mathcal{F}_{s}^e$ induces the following exact sequence
$$(*): 0 \to \mathcal{F}_{s}^e(-\sum_{t=1}^{t=d}H_t') \to \cdots \to \bigoplus_{t=1}^{t=d}\mathcal{F}_{s}^e(-H_t') \to \mathcal{J}\cdot\mathcal{F}_{s}^e \to 0$$
Take the cohomology of $(*)$. By a standard application of spectral sequence, from the condition $\clubsuit$ we deduce that if $s\geq d+1$ then for $i>0$
$$H^i(Y', \mathcal{J}\cdot\mathcal{F}_{s}^e(-iH')) = 0.$$
Let $\tau = \frac{\mathcal{I}_{y'}}{\mathcal{J}}\cdot\mathcal{F}_{s}^e(-iH')$. We have the following exact sequence
$$0 \to \mathcal{J}\cdot\mathcal{F}_{s}^e(-iH') \to \mathcal{I}_{y'}\cdot\mathcal{F}_{s}^e(-iH')  \to \tau \to 0.$$
Since $H^i(Y', \tau)=0$ for each $i >0$, by taking the cohomology of the above sequence, we can show that if $s\geq d+1$, then $$H^i(Y', \mathcal{I}_{y'}\cdot\mathcal{F}_{s}^e(-iH')) = 0 ~\mathrm{for~any}~i>0,$$
that is to say, the sheaf $\mathcal{I}_{y'}\cdot\mathcal{F}_{s}^e$ satisfies Mumford regularity.
\end{proof}


\subsection{} Recall that an \emph{irregular variety} $X$ is a smooth projective variety with $q(X):=\dim \mathrm{Pic}^0(X) >0$. For this kind of varieties, we can take advantage of the Albanese map and have the following theorem.
\begin{thm}\label{thm:bir-criterion-irr}
Let $X$ be a smooth projective variety with a morphism $a: X \to A$ to an abelian variety. Denote by $f: X \to Y$ the fibration arising from the Stein factorization of $a: X \to A$. Let $D, D_1,D_2$ be three divisors on $X$. Assume that $D$ is nef, big and $\mathbb{Q}$-Cartier.

(i) If $S^0_{-}(X_{\eta}, K_{X_{\eta}} + D_{\eta}) \neq 0$, then for any $\mathcal{P}_{\alpha} \in \mathrm{Pic}^0(A)$, $H^0(X, K_X + \ulcorner D\urcorner + a^*\mathcal{P}_{\alpha})) \neq 0$, and there exists some $\mathcal{P}_{\beta} \in \mathrm{Pic}^0(A)$ such that $S_{-}^0(X, K_X + \ulcorner D\urcorner + a^*\mathcal{P}_{\beta})\neq 0$.

(ii) Assume that $S^0_{-}(X_{\eta}, K_{X_{\eta}} + D_{\eta}) \neq 0$, $D_1$ is integral and that for any $\mathcal{P}_{\alpha} \in \mathrm{Pic}^0(A)$, $|D_1 + a^*\mathcal{P}_{\alpha}| \neq \emptyset$. Then for any $\mathcal{P}_{\alpha_0} \in \mathrm{Pic}^0(A)$, $S^0_{-}(X, K_X + D + D_1 + a^*\mathcal{P}_{\alpha_0}) \neq 0$.

(iii) Assume that $S^0_{-}(X_{\eta}, K_{X_{\eta}} + D_{\eta})$ is birational, both $D_1$ and $D_2$ are integral and that for any $\mathcal{P}_{\alpha} \in \mathrm{Pic}^0(A)$, $|D_i + a^*\mathcal{P}_{\alpha}| \neq \emptyset$. Then for any $\mathcal{P}_{\alpha_0} \in \mathrm{Pic}^0(A)$, $S^0_{-}(X, K_X + D + D_1+D_2 + a^*\mathcal{P}_{\alpha_0})$ is birational.

(iv) Assume that $S^0_{-}(X_{\eta}, K_{X_{\eta}} + D_{\eta})$ is birational, and that $D_1,D_2$ are nef and big $\mathbb{Q}$-Cartier $\mathbb{Q}$-divisors such that $S^0_{-}(X_{\eta}, K_{X_{\eta}} + (D_i)_{\eta}) \neq 0$. Then for any $\mathcal{P}_{\alpha_0} \in \mathrm{Pic}^0(A)$,
$S^0_{-}(X, K_X + D + (K_{X} + \ulcorner D_1\urcorner) + (K_{X} +  \ulcorner D_2\urcorner )+ a^*\mathcal{P}_{\alpha_0}))$ is birational.
\end{thm}

In \cite{Zhy14} the author proved that for a smooth projective variety $X$ of maximal Albanese dimension and of general type, if in addition the Albanese map is separable then $4K_X$ is birational. In general, if $K_X$ is big then there exists a nef and big $\mathbb{Q}$-divisor $D \leq K_X$, we can apply the theorem above and obtain the following result.
\begin{cor}
Let $X$ be a smooth projective variety of maximal Albanese dimension and of general type. Then $S^0_{-}(X, K_X + 5K_{X})$ is birational.
\end{cor}


Before preceding with the proof, let us recall some notions and results about Fourier-Mukai transform and generic vanishing sheaves developed by Pareschi and Popa \cite{PP03}.

Let $A$ be an abelian variety of dimension $d$, $\hat{A}= \mathrm{Pic}^0(A)$ and $\mathcal{P}$ the Poincar\'{e} line bundle on
$A \times \hat{A}$. Let $p,q$ denote the projections from $A \times \hat{A}$ to $A, \hat{A}$ respectively. The \emph{Fourier-Mukai transform} $R\Phi_{\mathcal{P}}: D^b(A) \rightarrow D^b(\hat{A})$ w.r.t. $\mathcal{P}$ is defined as
$$R\Phi_{\mathcal{P}}(-) := Rq_*(Lp^*(-)\otimes \mathcal{P})$$
which is a right derived functor. For a coherent sheaf $\mathcal{F}$ on $A$, let
$$D_A(\mathcal{F})= R\mathcal{H}om(\mathcal{F}, \mathcal{O}_A[d])~\mathrm{and}~\widehat{R\Delta(\mathcal{F})} = R\Phi_{\mathcal{P}}(D_A(\mathcal{F})).$$
\smallskip

\begin{defn}\label{defgv}
Given a coherent sheaf $\mathcal{F}$ on $A$, its ~$i$-$th$ ~$cohomological ~support ~locus$ is defined as
$$V^i(\mathcal{F}): = \{\alpha \in \hat{A}| h^i(\mathcal{F} \otimes \mathcal{P}_\alpha) > 0\}$$
The number $gv(\mathcal{F}): = min_{i>0}\{\mathrm{codim}_{\hat{A}}V^i(\mathcal{F}) - i\}$ is called the \emph{generic vanishing index} of $\mathcal{F}$, and
we say $\mathcal{F}$ is a $GV~ sheaf$ (resp. $M$-$regular~ sheaf$) if $gv(\mathcal{F}) \geq 0$ (resp. $>0$). If $V^i(\mathcal{F}) = \emptyset$ for any $i>0$ then we say $\mathcal{F}$ is an \emph{$IT^0$ sheaf}.
We say $\mathcal{F}$ is $continuously~ globally~ generated$ (CGG) if the sum of the evaluation maps
$$ev_{U,\mathcal{F}} : ~\oplus_{\alpha \in U}H^0(\mathcal{F} \otimes \mathcal{P}_\alpha) \otimes \mathcal{P}_\alpha^{-1} \rightarrow \mathcal{F}$$
is surjective for any dense subset $U \subset \mathrm{Pic}^0(A)$.
\end{defn}


\begin{prop}\label{prop:gv-M-reg}
Let $\mathcal{F}$ be a coherent sheaf on $A$.

(i) The sheaf $\mathcal{F}$ is a GV sheaf if and only if $\mathrm{codim}_{\hat{A}}(\mathrm{Supp}R^i\Phi_{\mathcal{P}}(\mathcal{F})) \geq i$ for any $i>0$, and if and only if $\widehat{R\Delta(\mathcal{F})}$ is quasi-isomorphic to a coherent sheaf on $\hat{A}$.

(ii) The sheaf $\mathcal{F}$ is M-regular if and only if $\mathrm{codim}_{\hat{A}}(\mathrm{Supp}R^i\Phi_{\mathcal{P}}(\mathcal{F})) > i$ for any $i>0$, and if and only if $\widehat{R\Delta(\mathcal{F})}$ is quasi-isomorphic to a torsion free coherent sheaf on $\hat{A}$.
\end{prop}
\begin{proof}
We refer the reader to \cite[Sec.2]{PP11} for the proof.
\end{proof}

\begin{lem}[\cite{BLNP}, Lemma 4.6]\label{lem:coker}
Let $\mathcal{F}$ be a GV-sheaf on $A$. Let $L$ be an ample line bundle on $A$. Then,
for all sufficiently large $n \in \mathbb{N}$, and for any subset $T \subseteq \mathrm{Pic}^0(A)$, the Fourier-Mukai transform $\Phi_{\mathcal{P}}$
induces a canonical isomorphism
$$H^0(A, \mathrm{coker}~ev_{T,\mathcal{F}}\otimes L^n) \cong (\mathrm{ker} ~\psi_{T,\mathcal{F}})^{\vee},$$
where $\psi_{T,\mathcal{F}}$ is the natural evaluation map defined as follows
$$\psi_{T,\mathcal{F}}: \mathrm{Hom}(\widehat{L^n},\widehat{R\Delta(\mathcal{F})}) \to \prod_{\alpha \in T}\mathcal{H}om(\widehat{L^n},\widehat{R\Delta(\mathcal{F})}) \otimes k(\alpha).$$
\end{lem}

\begin{cor}\label{cor:generation}
Let $\mathcal{F}$ be a coherent sheaf on $A$.

(i) If $\mathcal{F}$ is GV then there exist $\alpha_1, \alpha_2, \cdots, \alpha_m \in \hat{A}$ such that the evaluation homomorphism
$$\bigoplus_{i=1}^m H^0(A, \mathcal{F}\otimes \mathcal{P}_{\alpha_i}) \otimes \mathcal{P}_{\alpha_i}^{-1} \to \mathcal{F}$$
is surjective.

(ii) If $\mathcal{F}$ is $M$-regular, then it is CGG, in particular for any dense subset $V \subseteq \hat{A}$, there exist $\alpha_1, \alpha_2, \cdots, \alpha_m \in V$ such that the evaluation homomorphism
$$\bigoplus_{i=1}^m H^0(A, \mathcal{F}\otimes \mathcal{P}_{\alpha_i}) \otimes \mathcal{P}_{\alpha_i}^{-1} \to \mathcal{F}$$
is surjective.
\end{cor}
\begin{proof}
(i) Take an ample line bundle $L$ on $A$ and a sufficiently large number $n$ such that $\mathrm{coker}~ev_{\hat{A},\mathcal{F}}\otimes L^n$ is globally generated. As $L^n$ is an $IT^0$ sheaf, $\widehat{L^n}$ is a locally free sheaf on $\hat{A}$. Hence in Lemma \ref{lem:coker}, if setting $T=\hat{A}$, then $\mathrm{ker} ~\psi_{\hat{A},\mathcal{F}} =0$. This implies $H^0(A, (\mathrm{coker}~ev_{\hat{A},\mathcal{F}})\otimes L^n) =0$ by Lemma \ref{lem:coker}, hence $(\mathrm{coker}~ev_{\hat{A},\mathcal{F}})\otimes L^n = 0$, in other words,
$$ev_{\hat{A},\mathcal{F}}: \oplus_{\alpha \in \hat{A}}H^0(A, \mathcal{F} \otimes \mathcal{P}_\alpha) \otimes (\mathcal{P}_\alpha^{-1}) \rightarrow \mathcal{F}$$
is surjective. By Noetherian induction, we can find finitely many $\alpha_1, \alpha_2, \cdots, \alpha_m \in \hat{A}$ satisfying the requirement of (i).

(ii) If $\mathcal{F}$ is $M$-regular, then $\widehat{R\Delta(\mathcal{F})}$ is torsion free by Proposition \ref{prop:gv-M-reg}. Applying  Lemma \ref{lem:coker} again by setting $T=V$, since $T$ is dense, we see that $\mathrm{ker} ~\psi_{T,\mathcal{F}} =0$. We can prove (ii) by Noetherian induction as in (i).
\end{proof}

\begin{prop}\label{prop:tensor-of-gv} Let $\mathcal{F}$ and $\mathcal{E}$ be two coherent sheaves on $A$.
If $\mathcal{F}$ is GV and $\mathcal{E}$ is CGG, then the tensor product $\mathcal{F} \otimes \mathcal{E}$ is CGG.

\end{prop}
\begin{proof}
By Corollary \ref{cor:generation}, $\mathcal{F}$ is the quotient of $\bigoplus_{i=1}^m \mathcal{P}_{\alpha_i}$ for some $\alpha_1, \alpha_2, \cdots, \alpha_m \in \hat{A}$. Since $\bigoplus_{i=1}^m \mathcal{E}\otimes\mathcal{P}_{\alpha_i}$ is CGG, the tensor product $\mathcal{F} \otimes \mathcal{E}$, as the quotient of this sheaf, is CGG too.
\end{proof}

\begin{proof}[Proof of Theorem \ref{thm:bir-criterion-irr}]
Since $D$ is nef and big, by Proposition \ref{prop:F-stable-section} (iii) we can take $\Delta \geq 0$ such that
\begin{itemize}
\item[(a)] $S_{-}^0(X_{\eta}, K_{X_{\eta}} + D|_{X_{\eta}}) \subseteq S_{\Delta}^0(X_{\eta}, K_{X_{\eta}} + D|_{X_{\eta}})$ and
\item[(b)] $D- \Delta$ is ample, $p^g(D-\Delta)$ is integral for some $g$, and $S_{\Delta}^0(X, K_{X} + D + a^*\mathcal{P}_{\alpha})\subseteq S_{-}^0(X, K_{X} + D + a^*\mathcal{P}_{\alpha})$ for any $\alpha \in \mathrm{Pic}^0(A)$.
\end{itemize}
Let
$$\mathcal{F}^e= f_*F^{e}_*\mathcal{O}_X(K_X + \ulcorner p^e(D-\Delta)\urcorner)~\mathrm{and}~\mathcal{V}^e= \mathrm{Im}(f_*Tr^e:\mathcal{F}^e \to f_*\mathcal{O}_X(K_X + \ulcorner D\urcorner)).$$
By (a) we always have
\begin{itemize}
\item[(c)] $S^0_{-}(X_{\eta}, K_{X_{\eta}} + D_{\eta}) \subseteq \mathcal{V}^e \otimes k(\eta)$.
\end{itemize}
Let $\pi: Y \to A$ be the natural morphism such that $a=\pi\circ f$, which is finite. Applying  $R(\pi\circ f)_*= Ra_*$ to $F^{e}_*\mathcal{O}_X(K_X + \ulcorner p^e(D-\Delta)\urcorner)$ and considering the induced Lerray spectral sequence, by Fujita vanishing (Theorem \ref{thm:var-fujita-vanishing}), we can show that for sufficiently large $e$, $\pi_*\mathcal{F}^e$ is an $IT^0$ sheaf.
From now on  we fix a sufficiently divisible integer $e >0$ and assume that
\begin{itemize}
\item[(d)] $S_{\Delta}^0(X, K_{X} + D + a^*\mathcal{P}_{\alpha})=S_{\Delta}^{e}(X, K_{X} + D + a^*\mathcal{P}_{\alpha})$ for any $\mathcal{P}_{\alpha} \in \mathrm{Pic}^0(A)$, which is reasonable by Proposition \ref{prop:F-stable-section} (iii); and
\item[(e)] $\pi_*\mathcal{F}^e$ is an $IT^0$ sheaf.
\end{itemize}


(i) Assume $S^0_{-}(X_{\eta}, K_{X_{\eta}} + D_{\eta})\neq 0$. The assertion (c) implies that $\mathrm{rank}~\mathcal{V}^e >0$. And by (e), $\pi_*\mathcal{F}^e$ is CGG, so is $\pi_*\mathcal{V}^e$, in particular, for general $\alpha \in \hat{A}$ we have $H^0(A, \pi_*\mathcal{V}^e\otimes \mathcal{P}_{\alpha}) \neq 0$. Combining this with the fact that the set $\{\alpha \in \hat{A}|H^0(A, \pi_*\mathcal{V}^e\otimes \mathcal{P}_{\alpha}) = 0\}$ is open, which is a consequence of Lower Semicontinuity Theorem (\cite[Sec. 8.3]{FGA05}), we conclude that $H^0(A, \pi_*\mathcal{V}^e\otimes \mathcal{P}_{\alpha}) \neq 0$ for any $\alpha \in \hat{A}$, and thus $H^0(X, K_X+ \ulcorner D\urcorner + a^*\mathcal{P}_{\alpha})\neq 0$. Again since $\pi_*\mathcal{F}^e$ is CGG, there exists some $\beta \in \hat{A}$ such that the trace map
$$Tr^e: H^0(X, F^{e}_*\mathcal{O}_X(K_X + \ulcorner p^e(D-\Delta)\urcorner)\otimes a^*\mathcal{P}_{\beta}) \cong H^0(A, \pi_*\mathcal{F}^e\otimes \mathcal{P}_{\beta}) \to H^0(A, \pi_*\mathcal{V}^e\otimes \mathcal{P}_{\beta})$$
is a nonzero map, which implies by (b,d) that $S_{-}^0(X, K_{X} + D + a^*\mathcal{P}_{\beta}) \neq 0$.
\smallskip

(ii) Fix $\alpha_0 \in \hat{A}$. By (i) we can take $\beta \in \hat{A}$ such that $S_{-}^0(X, K_{X} + D + a^*\mathcal{P}_{\beta}) \neq 0$. By the assumption we are allowed to take a nonzero section $s \in H^0(X, D_1+ a^*\mathcal{P}_{\alpha_0-\beta})$. Then applying Proposition \ref{prop:F-stable-section} (i), we prove (ii) by
$$S_{-}^0(X, K_{X} + D + a^*\mathcal{P}_{\beta})\otimes s \subseteq S_{-}^0(X, K_{X} + D + D_1 + \mathcal{P}_{\alpha_0}).$$

\smallskip

(iii) Assume $S^0_{-}(X_{\eta}, K_{X_{\eta}} + D_{\eta})$ is birational. There exists a nonempty open subset $X'$ of $X$ such that,  $f^*\mathcal{V}^e \to \mathcal{O}_X(K_X + \ulcorner D\urcorner)$ is surjective over $X'$, and that for every closed point $x \in X'$, $f^*(\empty_x\mathcal{V}^e) \to \mathcal{I}_x\cdot\mathcal{O}_X(K_X + \ulcorner D\urcorner)$ is surjective over $X'$, here $\empty_x\mathcal{V}^e$ denotes the kernel of the natural map $\mathcal{V}^e \to f_*(k(x)\otimes \mathcal{O}_X(K_X + \ulcorner D\urcorner)) \cong k(f(x))$.

From now on fix a closed point $x \in X'$. We define $\empty_x\mathcal{F}^e$ by the following commutative diagram of exact sequences
{\small $$\xymatrix{
&0\ar[r] &\empty_x\mathcal{F}^e \ar[r]\ar[d]^{\gamma_x} &\mathcal{F}^e \ar[r]\ar[d]^{\gamma} &k(f(x)) \ar[r]\ar@{=}[d] &0\\
&0\ar[r] &\empty_x\mathcal{V}^e \ar[r]&\mathcal{V}^e \ar[r] &k(f(x)) \ar[r] &0}$$}
where $\gamma_x$ is surjective since $\gamma$ is.
\begin{lem}
$\pi_*(\empty_x\mathcal{F}^e)$ is a GV-sheaf.
\end{lem}
\begin{proof}
Since $\pi$ is a finite morphism, the following sequence is exact
$$0 \to \pi_*(\empty_x\mathcal{F}^e) \to \pi_*\mathcal{F}^e \to k(a(x))\cong k \to 0.$$
Tensoring the above sequence with $\mathcal{P}_{\alpha} \in \mathrm{Pic}^0(A)$ and taking cohomology we obtain a long exact sequence, then since $\pi_*\mathcal{F}^e$ is $IT^0$ we deduce that
\begin{itemize}
\item
for~$i\geq 2$, $H^i(A, (\pi_*(\empty_x\mathcal{F}^e))\otimes \mathcal{P}_{\alpha}) =0$~which means $V^i(\pi_*(\empty_x\mathcal{F}^e)) = \emptyset$, and
\item
$H^1(A, \pi_*(\empty_x\mathcal{F}^e)\otimes \mathcal{P}_{\alpha})= \mathrm{coker} (H^0(A, \pi_*\mathcal{F}^e\otimes \mathcal{P}_{\alpha}) \to H^0(A, k(a(x))\otimes \mathcal{P}_{\alpha})\cong k)$.
\end{itemize}
Then since $\pi_*\mathcal{F}^e \to k(a(x))$ is surjective and $\pi_*\mathcal{F}^e$ is CGG, there exists $\alpha \in \hat{A}$ such that the evaluation map $H^0(A, \pi_*\mathcal{F}^e\otimes \mathcal{P}_{\alpha}) \to k(a(x))$ is surjective. Therefore, the closed subset
$$\{\alpha \in \hat{A}|\mathrm{the ~evaluation ~map}~H^0(A, \pi_*\mathcal{F}^e\otimes \mathcal{P}_{\alpha}) \to k~\mathrm{is ~zero}\}$$
has codimension at least one in $\hat{A}$, thus $\pi_*(\empty_x\mathcal{F}^e)$ is a GV-sheaf.
\end{proof}

By Corollary \ref{cor:generation}, there exist $\alpha_1, \alpha_2, \cdots, \alpha_m \in \hat{A}$ such that the evaluation homomorphism
$$\bigoplus_{j=1}^m H^0(A, \pi_*(\empty_x\mathcal{F}^e)\otimes \mathcal{P}_{\alpha_j}) \otimes \mathcal{P}_{\alpha_j}^{-1} \to \pi_*(\empty_x\mathcal{F}^e)$$
is surjective. Notice that for $\alpha \in \mathrm{Pic}^0(A)$, by (d) the image of the composition map
\begin{align*}
H^0(A, \pi_*(\empty_x\mathcal{F}^e)\otimes \mathcal{P}_{\alpha}) &\hookrightarrow H^0(A, \pi_*\mathcal{F}^e\otimes \mathcal{P}_{\alpha}) \\
&\cong H^0(X, \mathcal{O}_X(K_X + \ulcorner p^e(D-\Delta)\urcorner)+ p^ea^*\mathcal{P}_{\alpha}) \\
&\xrightarrow{Tr^e} H^0(X, \mathcal{O}_X(K_X + \ulcorner D\urcorner + a^*\mathcal{P}_{\alpha})
\end{align*}
coincides with $\empty_xS^0_{\Delta}(X, K_X + D + a^*\mathcal{P}_{\alpha})$ (the subspace of $S^0_{\Delta}(X, K_X + D + a^*\mathcal{P}_{\alpha})$ consisting of those sections vanishing at $x$), and that
$$a^*(\pi_*(\empty_x\mathcal{F}^e)) \xrightarrow{Tr^e} f^*(\empty_x\mathcal{V}^e) \to \mathcal{I}_x\cdot\mathcal{O}_X(K_X + \ulcorner D\urcorner)$$
is surjective on $X'$. Then we can conclude
$$\bigoplus_{j=1}^m (\empty_xS^0_{\Delta}(X, K_X + D + a^*\mathcal{P}_{\alpha_j})\otimes a^*\mathcal{P}_{-\alpha_j} )\to \mathcal{I}_x\cdot\mathcal{O}_X(K_X + \ulcorner D\urcorner)$$
is surjective on $X'$.

Now fix $\alpha_0 \in \hat{A}$. We can take a nonempty open subset $\hat{V} \subseteq \hat{A}$ such that for $i=1,2$, $\Phi_{\mathcal{P}}(a_*\mathcal{O}_X(D_i))$ is locally free over $\hat{V}$ and for any $\alpha \in \hat{V}$  $$\Phi_{\mathcal{P}}(a_*\mathcal{O}_X(D_i))\otimes k(\alpha) \cong H^0(X, \mathcal{O}_X(D_i)\otimes a^*\mathcal{P}_{\alpha}).$$
Let $T_i = \cap_{\alpha \in \hat{V}} \mathrm{Bs}|D_i + a^*\mathcal{P}_{\alpha}|$ and $X'' = X' \setminus (T_1 \cup T_2)$.
Observe this fact
\begin{itemize}
\item[(f)] for a closed point $z \in X''$ the set
$\hat{V}_i=\{\alpha \in \hat{V}| z\notin \mathrm{Bs}|D_i + a^*\mathcal{P}_{\alpha}|\}$
is a nonempty open subset of $\hat{V}$.
\end{itemize}
To prove that $S^0_{-}(X, K_X + \ulcorner D\urcorner + D_1 + D_2 + a^*\mathcal{P}_{\alpha_0})$ is birational, we only need to verify that for any $x \in X''$, $\mathcal{I}_x\cdot\mathcal{O}_X(K_X + \ulcorner D\urcorner + D_1 + D_2 + a^*\mathcal{P}_{\alpha_0})$ is globally generated over $X''$ by $\empty_xS^0_{-}(X, K_X + \ulcorner D\urcorner + D_1 + D_2 + a^*\mathcal{P}_{\alpha_0})$. Fix a closed point $z \in X''$. Then by (f), we know that for general $\beta_1, \cdots, \beta_m \in \hat{V}$,
$$z \notin\mathrm{Bs}|D_1 + a^*\mathcal{P}_{\alpha_0 - \alpha_j - \beta_j}| \cup \mathrm{Bs}|D_2 + a^*\mathcal{P}_{\beta_j}|,~j=1,2, \cdots, m.$$
For any $j =1,2, \cdots, m$, we may take
$$s_j \in H^0(X, D_1 +  a^*\mathcal{P}_{\alpha_0 - \alpha_j - \beta_j})~\mathrm{and}~t_j \in H^0(X, D_2 +  a^*\mathcal{P}_{\beta_j})$$
such that $s_j(z) \neq 0$ and  $t_j(z) \neq 0$. Then
$$\empty_xS^0_{\Delta}(X, K_X + D + a^*\mathcal{P}_{\alpha_j})\otimes s_j\otimes t_j \subseteq~ \empty_xS^0_{\Delta}(X, K_X + D + D_1 +D_2 +a^*\mathcal{P}_{\alpha_0}).$$
It follows that around $z$, $\sum_{j}~\empty_xS^0_{\Delta}(X, K_X + D + a^*\mathcal{P}_{\alpha_j})\otimes s_j\otimes t_j$ generates the sheaf $\mathcal{I}_x\cdot\mathcal{O}_X(K_X + \ulcorner D\urcorner + D_1 + D_2 + a^*\mathcal{P}_{\alpha_0})$. This finish the proof of the assertion (iii).
\smallskip

(iv) Applying the assertion (i), the assumptions on $D_i$ imply that for any $\alpha \in \hat{A}$, $H^0(X, K_X +\ulcorner D_i\urcorner + a^*\mathcal{P}_{\alpha}) \neq 0$. Then applying the assertion (iii) we get (iv).
\end{proof}

\section{Frobenius stable pluricanonical systems on curves and surfaces}\label{sec:dim2}
In this section we will study Frobenius stable pluricanonical systems on curves and surfaces in positive characteristic.

\begin{thm}\label{thm:eff-curve}
Let $X$ be a normal projective curve over an algebraically closed field $k$ of characteristic $p$. Let $D$ be a $\mathbb{Q}$-Cartier $\mathbb{Q}$-divisor on $X$. If $\deg D >1$ then $S_{-}^0(X, K_X + D)$ is base point free; and if $\deg D >2$ then $S_{-}^0(X, K_X + D)$ is very ample.
\smallskip

In particular, if $g(X) \geq 2$ then $S_{-}^0(X, K_X + K_X) \neq 0$ and $S_{-}^0(X, K_X + nK_X)$ is very ample for $n\geq 2$.
\end{thm}
\begin{proof} We may replace $D$ with a smaller divisor, which still satisfies the conditions of the theorem and in addition that the index is a power of $p$. So we only need to prove the assertion for $S^0(X, K_X + D)$ by Proposition \ref{prop:F-stable-section} (iii).

By direct local computations, we can show that
\begin{itemize}
\item
the trace map $Tr^e: F^e_*\mathcal{O}_X(K_X + \ulcorner p^eD\urcorner) \to \mathcal{O}_X(K_X + \ulcorner D\urcorner)$ is surjective, and for a closed point $x \in X$, the kernel of the composition map
$$F^e_*\mathcal{O}_X(K_X + \ulcorner p^eD\urcorner) \to \mathcal{O}_X(K_X + \ulcorner D\urcorner) \to k(x)\otimes \mathcal{O}_X(K_X + \ulcorner D\urcorner) \cong k(x)$$
is $F^e_*\mathcal{O}_X(K_X + \ulcorner p^eD\urcorner - q_{e,x} x)$ for some integer $q_{e,x} \in [1,p^e]$.
\end{itemize}
It follows that for any closed point $x \in X$, the following commutative diagram of exact sequences holds
{\small$$\xymatrix@C=0.5cm{
&0\ar[r] &F^e_*\mathcal{O}_X(K_X + \ulcorner p^eD\urcorner - q_{e,x}x) \ar[r]\ar[d] &F^e_*\mathcal{O}_X(K_X + \ulcorner p^eD\urcorner - (q_{e,x}-1)x)\ar[r]\ar[d] &k(x) \ar[r]\ar@{=}[d] &0\\
&0\ar[r] &\mathcal{O}_X(K_X + \ulcorner D\urcorner -x) \ar[r]\ar[d] &\mathcal{O}_X(K_X + \ulcorner D\urcorner) \ar[r]\ar[d] &k(x) \ar[r] &0\\
&  &0  &0 &  &
}$$}

First assume $\deg D >1$. Then for for any closed point $x \in X$ and any sufficiently large $e$, $H^1(X, F^e_*\mathcal{O}_X(K_X + \ulcorner p^eD\urcorner - q_{e,x}x) = 0$. By taking cohomology of the first row of the diagram above, we show that $\mathcal{O}_X(K_X + \ulcorner D\urcorner)$ is globally generated by $S^0(X, K_X + D)$ at $x$. This proves $S^0(X, K_X + D) \neq 0$ is base point free.

Next assume $\deg D >2$. Fix a closed point $z \in X$ and set $D'= D-z$. By replacing $D$ with $D'$, the argument in the previous paragraph shows that the sheaf $\mathcal{O}_X(K_X + \ulcorner D\urcorner - z)$ is globally generated by $S^0(X, K_X + D-z)$.
From this we conclude that $S^0(X, K_X + D)$ is very ample.

\smallskip
Finally the remaining assertion follows from the fact that $\deg K_X \geq 2$ if $X$ is of general type.
\end{proof}

\begin{thm}\label{thm:eff-surface}
Let $X$ be a smooth surface over an algebraically closed field $k$ of characteristic $p$. Let $a: X \to A$ be the Albanese map (which is trivial if $q(X)=0$), and let $f:X \to Y$ be the fibration induced by the Stein factorization of $a$. Let $D$ be a nef and big Cartier divisor and $D'$ a big $\mathbb{Q}$-Cartier $\mathbb{Q}$-divisor on $X$.
\smallskip

Case (i) $q(X) =0$. Let $r$ be a natural number such that $(rD-K_X)\cdot D \geq 0$. Then $S_{-}^0(X, K_X + (r+2)D + D') \neq 0$, and $S_{-}^0(X, K_X + (2r+4)D + D')$ is birational.
\smallskip

Case (ii) $q(X) >0$. Fix $\mathcal{P}_\alpha \in \mathrm{Pic}^0(A)$.

Case (ii-1) $\dim Y =2$. Then $S_{-}^0(X, K_X + D' + (K_X+  D) + a^*\mathcal{P}_\alpha) \neq 0$, and $S_{-}^0(X, K_X + D'+ 2(K_X+ D) + a^*\mathcal{P}_\alpha)$ is birational.

Case (ii-2) $\dim Y =1$. If $\deg D|_{X_{\eta}} \geq 2$, then $S_{-}^0(X, K_X + D + D' + (K_X+ D) + a^*\mathcal{P}_\alpha) \neq 0$, and $S_{-}^0(X, K_X + D + D' + 2(K_X+ D) + a^*\mathcal{P}_\alpha)$ is birational; if $\deg D|_{X_{\eta}} = 1$, then $S_{-}^0(X, K_X + D + D' + (K_X+ 2D)+ a^*\mathcal{P}_\alpha) \neq 0$, and $S_{-}^0(X, K_X + 2D + D' + 2(K_X+ 2D)+ a^*\mathcal{P}_\alpha)$ is birational.

\smallskip
Moreover, if $X$ is of general type then $S_{-}^0(X, K_X + nK_X) \neq 0$ for $n \geq 4$, and $S_{-}^0(X, K_X + nK_X)$ is birational for $n\geq 7$.
\end{thm}
\begin{proof} We may replace $D'$ with a smaller divisor to assume it is nef.
\smallskip

Case (i). In this case we have $h^1(\mathcal{O}_X) - h^2(\mathcal{O}_X) \leq q(X) =0$ (see for example \cite[Remark 9.5.15, 9.5.25]{FGA05}), thus $\chi(\mathcal{O}_X) >0$. And by the assumption $(rD-K_X)\cdot D \geq 0$, we have $h^2(X,(r+1)D) = h^0(X, K_X-(r+1)D)=0$. Applying Riemann-Roch formula we get
$$h^0(X,(r+1)D) \geq \chi(\mathcal{O}_X((r+1)D)) = \frac{((r+1)D-K_X)\cdot (r+1)D}{2} + \chi(\mathcal{O}_X) \geq 2.$$
Remark that if we blow up $X$ and replace $D$ with the pullback, it still holds that $(rD-K_X )\cdot D \geq 0$, hence we are allowed to blow up $X$ in the proof. Take a linear system $|V| \subseteq |(r+1)D|$ of dimension one. If necessary blowing up $X$, we may assume $|V|= |H| + E$ where $|H|$ is free of base point and defines a fibration $f: X \to Y=\mathbb{P}^1$. Since
$\dim Y=1$, by results of Sec. \ref{sec:bc}, $X_{\bar{\eta}}$ is integral. Denote by $Z_{\bar{\eta}}$ the normalization of $X_{\bar{\eta}}$. Since $\deg (D + D')|_{Z_{\bar{\eta}}} >1$ and $\deg (2D + D')|_{Z_{\bar{\eta}}} >2$, by Theorem \ref{thm:eff-curve} we know that $S_{-}^0(Z_{\bar{\eta}}, K_{Z_{\bar{\eta}}} + (D + D')|_{Z_{\bar{\eta}}}) \neq 0$ and $S_{-}^0(Z_{\bar{\eta}}, K_{Z_{\bar{\eta}}} + (2D + D')|_{Z_{\bar{\eta}}})$ is birational. Combining Theorem \ref{thm:bir-geo-gen} and \ref{thm:bir-criterion-1}, it follows  that $S_{-}^0(X, K_X + H + D+ D') \neq 0$ and $S_{-}^0(X, K_X + 2H + 2D+ D')$ is birational. This statement still holds if $H$ replaced with $(r+1)D$ by Proposition \ref{prop:F-stable-section} (i).
\smallskip

Case (ii). In this case, we denote by $Z_{\bar{\eta}}$ the normalization of $X_{\bar{\eta}}^{\mathrm{red}}$.

If $\dim Z_{\bar{\eta}} = 0$ then both $S_{-}^0(Z_{\bar{\eta}}, K_{Z_{\bar{\eta}}} + D|_{Z_{\bar{\eta}}})$ and $S_{-}^0(Z_{\bar{\eta}}, K_{Z_{\bar{\eta}}} + D'|_{Z_{\bar{\eta}}})$ are birational. Applying Theorem  \ref{thm:bir-criterion-irr} (i), (ii) and (iv), we can show the claimed result in Case (ii-1).

Now assume $\dim Z_{\bar{\eta}} = 1$. We know that $X_{\bar{\eta}}$ is integral, which implies $\deg D|_{X_{\eta}} = \deg D|_{Z_{\bar{\eta}}}$. Applying Theorem \ref{thm:eff-curve} we obtain that
\begin{itemize}
\item
$S_{-}^0(Z_{\bar{\eta}}, K_{Z_{\bar{\eta}}} + 2D|_{Z_{\bar{\eta}}}) \neq 0$, $S_{-}^0(Z_{\bar{\eta}}, K_{Z_{\bar{\eta}}} + (D+D')|_{Z_{\bar{\eta}}}) \neq 0$, and if moreover $\deg D|_{X_{\eta}} \geq 2$ then $S_{-}^0(Z_{\bar{\eta}}, K_{Z_{\bar{\eta}}} + D|_{Z_{\bar{\eta}}}) \neq 0$; and
\item
$S_{-}^0(Z_{\bar{\eta}}, K_{Z_{\bar{\eta}}} + (2D + D')|_{Z_{\bar{\eta}}})$ is birational, and if $\deg D|_{X_{\eta}} \geq 2$ then $S_{-}^0(Z_{\bar{\eta}}, K_{Z_{\bar{\eta}}} + (D + D')|_{Z_{\bar{\eta}}})$ is birational.
\end{itemize}
Theorem \ref{thm:bir-geo-gen} tells that the corresponding results hold true for $X_{\eta}$. Then we can apply Theorem \ref{thm:bir-criterion-irr} (i), (ii) and (iv) to show the assertion in Case (ii-2).
\smallskip

At the end, let us focus on the pluricanonical system. We assume $X$ is of general type and take a minimal model $\bar{X}$, which is endowed with the birational morphism $\rho:X \to \bar{X}$. Set $D = D'=\rho^*K_{\bar{X}}$. Since $K_X \geq \rho^*K_{\bar{X}}$, it is enough to show the corresponding assertions for $S^0(X, K_X + n\rho^*K_{\bar{X}})$. If $q(X) =0$, then applying results of Case (i) we obtain that $S_{-}^0(X, K_X + n\rho^*K_{\bar{X}}) \neq 0$ for $n\geq 4$, and $S_{-}^0(X, K_X + n\rho^*K_{\bar{X}})$ is birational for $n\geq 7$. If $q(X) >0$ and $\dim Y = 2$, then applying results of Case (ii-1) we obtain that $S_{-}^0(X, K_X + n\rho^*K_{\bar{X}}) \neq 0$ for $n\geq 3$, and $S_{-}^0(X, K_X + n\rho^*K_{\bar{X}})$ is birational for $n\geq 5$. If $q(X) >0$ and $\dim Y = 1$, since the arithmetic genus $p_a(X_{\eta}) \geq 2$ which implies $\deg \rho^*K_{\bar{X}}|_{X_{\eta}}= \deg K_X|_{X_{\eta}} \geq 2$, then applying results of Case (ii-2) we obtain that $S_{-}^0(X, K_X + n\rho^*K_{\bar{X}}) \neq 0$ for $n\geq 4$, and $S_{-}^0(X, K_X + n\rho^*K_{\bar{X}})$ is birational for $n\geq 6$.
\end{proof}

\begin{cor}\label{cor:eff-surface}
Let the notation be as in Theorem \ref{thm:eff-surface}. Assume that the Albanese map $a: X \to A$ factors through a birational contraction $\sigma: X \to \bar{X}$ such that $D = \sigma^*\bar{D}$ is the pull back of a nef and big Cartier divisor $\bar{D}$ on $\bar{X}$.
Assume moreover that there exist an integer $r>0$ and an integral effective $\sigma$-exceptional divisor $E$ such that $D''=rD +E - K_X$ is big. Then for any $\mathcal{P}_{\alpha} \in \mathrm{Pic}^0(A)$, $S_{-}^0(X, K_X + (2+r)D + D'+ a^*\mathcal{P}_{\alpha}) \neq 0$, and $S_{-}^0(X, K_X + (4+2r)D + D'+ a^*\mathcal{P}_{\alpha})$ is birational.
\end{cor}
\begin{proof}
By the assumption we have $(rD - K_X)\cdot D =D''\cdot D > 0$, when $q(X)=0$ we can apply the first case of Theorem \ref{thm:eff-surface}.

From now on we assume $q(X) >0$ and fix $\mathcal{P}_{\alpha} \in \mathrm{Pic}^0(A)$. We only consider the case $\dim X_{\eta} =1$, because the other case is similar and easier.

We claim that  for any
$\mathcal{P}_\beta \in \mathrm{Pic}^0(A)$, $H^0(X, (r+1)D  + a^*\mathcal{P}_\beta) \neq 0$. To prove this, we take a nef and big divisor $D''_{-}$ such that $D''_{-} \leq D'' = \ulcorner D''\urcorner$. By combining Theorem \ref{thm:eff-curve} with Theorem \ref{thm:bir-geo-gen}, we can show $S_{-}^0(X_{\eta}, K_{X_{\eta}} + (D + D''_{-})_{\eta}) \neq 0$. Then Theorem \ref{thm:bir-criterion-irr} (i) tells that for any
$\mathcal{P}_\beta \in \mathrm{Pic}^0(A)$, $H^0(X, K_X + D + \ulcorner D''_{-} \urcorner+ a^*\mathcal{P}_\beta) \neq 0$, thus $H^0(X, K_X + D + D'' + a^*\mathcal{P}_\beta) \neq 0$, namely, $H^0(X, (r+1)D + E + a^*\mathcal{P}_\beta)\neq 0$. The claimed nonvanishing follows from the projection formula since $E$ is $\sigma$-exceptional.

Applying Theorem \ref{thm:eff-curve} and \ref{thm:bir-geo-gen} again, we have that
\begin{itemize}
\item[(a)] $S_{-}^0(X_{\eta}, K_{X_{\eta}} + (D + D')_{\eta}) \neq 0$; and
\item[(b)] $S_{-}^0(X_{\eta}, K_{X_{\eta}} + (2D + D')_{\eta})$ is birational.
\end{itemize}
Combining these with the claim above, we can prove the two desired assertions by applying Theorem \ref{thm:bir-criterion-irr} (ii) and (iii) respectively.
\end{proof}

To do induction we need to study the pluricanonical systems on the generic fiber, which is defined over a non-algebraically closed field. To study threefolds, we shall need the following result.
\begin{thm}\label{thm:eff-surface-nonclosed-field}
Let $X$ be a minimal regular surface of general type over an $F$-finite field $K$ and $D'$ a big $\mathbb{Q}$-Cartier $\mathbb{Q}$-divisor on $X$. Then $S_{-}^0(X, K_X + 4K_X + D') \neq 0$, and $S_{-}^0(X, K_X + 8K_X + D')$ is birational; and if moreover $p>2$ then $S_{-}^0(X, K_X + 3K_X + D') \neq 0$, and $S_{-}^0(X, K_X + 6K_X + D')$ is birational.
\end{thm}
\begin{proof}
Let $Z_{\bar{K}}$ denote the normalization of $X_{\bar{K}}^{\mathrm{red}}$. By results of Section \ref{sec:canonical-bundle-bc} we may write that $K_X|_{Z_{\bar{K}}} = K_{Z_{\bar{K}}} + (p-1)C$ where $C \geq 0$ is a Weil divisor arising from the conductors. Let $\rho: W_{\bar{K}} \to Z_{\bar{K}}$ be a minimal smooth resolution of singularities. Then $\rho^*K_{Z_{\bar{K}}} \geq K_{W_{\bar{K}}}$, and in particular $K_X|_{W_{\bar{K}}} \geq K_{W_{\bar{K}}} + \rho^*(p-1)C$. Applying Theorem \ref{thm:bir-geo-gen} and Proposition \ref{prop:F-stable-section} (iv), we only need to prove the corresponding assertions for $S_{-}(W_{\bar{K}}, K_{W_{\bar{K}}} + (rK_X + D')|_{W_{\bar{K}}})$.

From the construction above, we can write that
\begin{align}\label{eq:***}~K_X|_{W_{\bar{K}}} \sim K_{W_{\bar{K}}} + (p-1)C' + E\end{align}
where $C'$ is an effective integral divisor on $W_{\bar{K}}$ and $E$ is an effective $\rho$-exceptional divisor. It follows that for $r_1, r_2\geq 0$, there is a natural inclusion
$$S_{-}(W_{\bar{K}}, K_{W_{\bar{K}}} + r_1K_{W_{\bar{K}}} + (r_2K_X + D')|_{W_{\bar{K}}}) \hookrightarrow S_{-}(W_{\bar{K}}, K_{W_{\bar{K}}} + ((r_1+r_2)K_X + D')|_{W_{\bar{K}}}).$$
\smallskip

Case (i) $q(W_{\bar{K}}) = 0$. We set $D= K_X|_{W_{\bar{K}}}$ and apply Theorem \ref{thm:eff-surface}, then we obtain that $S_{-}(W_{\bar{K}}, K_{W_{\bar{K}}} + (3K_X + D')|_{W_{\bar{K}}}) \neq 0$ and  $S_{-}(W_{\bar{K}}, K_{W_{\bar{K}}} + (6K_X + D')|_{W_{\bar{K}}})$ is birational.
\smallskip

Case (ii) $q(W_{\bar{K}}) \geq 1$. The Albanese map of $W_{\bar{K}}$ induces a fibration $f: W_{\bar{K}} \to Y$ to a normal variety $Y$ over $\bar{K}$ by the Stein factorization.

Case (ii-1) $\dim Y =2$. We set $D= K_X|_{W_{\bar{K}}}$ and apply Theorem \ref{thm:eff-surface}, then we obtain that $S_{-}(W_{\bar{K}}, K_{W_{\bar{K}}} + D'|_{W_{\bar{K}}}+ K_{W_{\bar{K}}} +K_X|_{W_{\bar{K}}}) \neq 0$
and $S_{-}(W_{\bar{K}}, K_{W_{\bar{K}}} + D'|_{W_{\bar{K}}} + 2(K_{W_{\bar{K}}} +K_X|_{W_{\bar{K}}}))$ is birational, which infer that $S_{-}(W_{\bar{K}}, K_{W_{\bar{K}}} + (2K_X + D')|_{W_{\bar{K}}}) \neq 0$
and that $S_{-}(W_{\bar{K}}, K_{W_{\bar{K}}} + (4K_X + D')|_{W_{\bar{K}}})$ is birational.

Case (ii-2) $\dim Y =1$. We may set $D= K_X|_{W_{\bar{K}}}$ and apply  Theorem \ref{thm:eff-surface}, by similar argument we can show $S_{-}(W_{\bar{K}}, K_{W_{\bar{K}}} + (4K_X + D')|_{W_{\bar{K}}}) \neq 0$, and $S_{-}(W_{\bar{K}}, K_{W_{\bar{K}}} + (8K_X + D')|_{W_{\bar{K}}})$ is birational.
\smallskip

For the remaining assertion, we may assume $p>2$ and only need to consider Case (ii-2).
Let $F$ be the normalization of geometric generic fiber of $f$. Since each irreducible component of $E$ is a rational curve by  Propsition \ref{prop:ex-over-surface}, we have $E\cdot F =0$.
Then by the formula (\ref{eq:***}), it follows that
$$\deg K_X|_F = \deg K_{W_{\bar{K}}}|_F + (p-1) C'\cdot F >0.$$
Since the number $\deg K_{W_{\bar{K}}}|_F \geq -2$ and is even, we see that if $p\geq 3$ then $\deg K_X|_F \geq 2$. Therefore, we can complete the proof by applying Theorem \ref{thm:eff-surface} again.
\end{proof}

\section{A Miyaoka-Yau type inequality for minimal threefolds in positive characteristic}\label{sec:my-ineq}
In practice, we require $\mathcal{O}_X(nK_X)$ to be big to obtain enough global sections of $\mathcal{O}_X(nK_X)$ for a minimal threefold. When applying Riemann-Roch, we need to compute $K_Z \cdot \rho^*c_2(X)$ where $\rho: Z \to X$ is a smooth resolution. In characteristic zero, for a smooth variety $X$ with pseudo-effective $K_X$, Miyaoka \cite{Miy85} proved that $\Omega_X^1$ satisfies generic positivity, which implies $c_2(X)$ is pseudo-effective. However in characteristic $p>0$, this is not necessarily true, as is known, there are examples of Raynaud's surfaces, which are of general type and have $c_2(X) <0$.
In \cite{XZ19} the authors studied the minimal threefolds in characteristic $p \geq 5$ with $\nu(X, K_X) \leq 2$ and $K_Z \cdot \rho^*c_2(X) < 0$, and proved that $X$ is uniruled and $K_X$ is semiample. When $K_X$ is big, we prove the following result.

\begin{thm}\label{thm:my-ineq}
Let $X$ be a minimal, terminal, projective threefold of general type over an algebraically closed field $k$ of characteristic $p>0$. Let $\rho: Z \to X$ be a smooth resolution. Then
\begin{align}\label{eq:my-ineq} c_2(Z) \cdot \rho^*K_X + AK_X^3 \geq 0\end{align}
where $A = \frac{(54n_0^2 + 9n_0)p^2 + (9n_0 +\frac{3}{2})p}{(p-1)^2}$ and $n_0$ is the Cartier index of $K_X$.
\end{thm}

Before the proof, let us recall some related results. For a smooth minimal surface $X$ of general type over an algebraically closed field $k$, if char $k=0$, then we have the famous Miyaoka-Yau inequality $c_2(X) - 3K_X^2 \geq 0$ (\cite{Miy77}); if char $k=p>0$, \cite{GSZ19} proved
\begin{align}\label{eq:my-ineq-sur}c_2(X) + \frac{5}{8}c_1(X)^2 \geq 0.\end{align}
This inequality is implied by Xiao's slope inequality and is sharp! We may regard the inequality (\ref{eq:my-ineq}) as an analogue of (\ref{eq:my-ineq-sur}) in dimension three and call them \emph{Miyaoka-Yau type inequalities}. Our proof follows Miyaoka's approach \cite{Miy85}, so the main ingredients include the stability of the tangent bundle and Bogomolov inequality proved by \cite{Lan04} in positive characteristic.
\smallskip

From now on, we work over an algebraically closed field $k$ of characteristic $p>0$.
\subsection{Stability}
Let $X$ be a smooth projective variety of dimension $n$, $E$ a torsion free coherent sheaf and $D_1, \cdots, D_{n-1}$ nef $\mathbb{R}$-Cartier divisors on $X$. Assume that $D_1\cdots D_{n-1}$ is nontrivial. Recall that the \emph{slope} of $E$ with respect to $(D_1,...,D_{n-1})$ is defined as
$$\mu(E) = \frac{c_1(E)\cdot D_1 \cdots D_{n-1}}{\rk~E},$$
and $E$ is \emph{semistable} (resp. \emph{stable}) w.r.t. $(D_1,...,D_{n-1})$ if for every nontrivial subsheaf $E' \subset E$
$$\mu(E') \leq \mu(E)~\mathrm{(}\mathrm{resp}. <\mathrm{)}.$$
It is well known that $E$ has \emph{Harder-Narasimhan (HN) filtration} (\cite[Sec. I.1.3]{HL10})
$$0 = E_0 \subset E_1 \subset E_2 \subset \cdots \subset E_m = E$$
such that the $F_i=E_i/E_{i-1}$'s are semistable and $\mu(F_i) >\mu(F_{i+1})$.

Denote by $F: X \cong X^1 \to X$ the Frobenius map. Recall that $F^*E$ has a canonical connection $\nabla_{\mathrm{can}}: F^*E \to F^*E\otimes \Omega_{X^1}^1$ which is determined by
$$\mathrm{for}~ a \in \O_{X^1}~\mathrm{and}~e \in E,~~ \nabla_{\mathrm{can}}: a\otimes e \mapsto da \otimes e,$$
and it yields the Cartier  equivalence of categories between
the category of quasi-coherent sheaves on $X$ and the category of quasi-coherent
$\mathcal{O}_{X^1}$-modules with integrable $k$-connections, whose $p$-curvature is zero (cf. \cite[Theorem 5.1]{Kat70}).

It is well known that in positive characteristic the stability is not preserved by the Frobenius pullback. We say that $E$ is \emph{strongly semistable} (resp. \emph{strongly stable}) w.r.t. $(D_1,...,D_{n-1})$ if for every integer $k\geq 0$ the pullback $F^{k*}E$ is semistable (resp. stable). We need to consider the Harder-Narasimhan filtrations of the $F^{k*}E$'s. It is worth mentioning that \emph{$E$ satisfies the property ``fdHN''}: there exists some $k_0$ such that for any $k\geq k_0$ the Harder-Narasimhan filtration of $F^{k*}E$ coincides with the pull-back of that of $F^{k_0*}E$ (\cite[Theorem 2.7]{Lan04}).
Recall the invariants
$$L_{\mathrm{max}}(E) = \lim_{k \to \infty} \frac{\mu_{\max}(F^{k*}(E))}{p^k}~\mathrm{and}~L_{\mathrm{min}}(E) = \lim_{k \to \infty} \frac{\mu_{\min}(F^{k*}(E))}{p^k}$$
and denote $\alpha(E)=\max\{L_{\max}(E)-\mu_{\max}(E), ~\mu_{\min}(E)-L_{\min}(E)\}$.

\subsection{The gap between the slopes}
It is more convenient to consider the stability with respect to ample divisors. The natural method is to perturb the $D_i$'s.  Let $H$ be an ample divisor and $\epsilon >0$. Consider the stability with respect to $(D_1 + \epsilon H,\cdots, D_{n-1} + \epsilon H)$. Let $\mu_{i_1, \cdots, i_k; ~H}$ denote the slope with respect to $(D_{i_1},...,D_{i_k}, H, \cdots, H)$ and $\mu_{\epsilon H}$ denote the slope with respect to $(D_1 + \epsilon H,\cdots, D_{n-1} + \epsilon H)$. Then
$$\mu_{\epsilon H}= \mu + \sum_{k <n-1}\epsilon^{n-1-k} \sum_{i_1, \cdots, i_k}\mu_{i_1, \cdots, i_k; ~H}.$$
If $\mathcal{G}_0$ is the maximal $\mu$-slope subsheaf of $E$ and $\mathcal{G}_{\epsilon}$ is the maximal $\mu_{\epsilon H}$-slope subsheaf, then
$$\mu_{\epsilon H}(\mathcal{G}_0) \leq \mu_{\epsilon H, \max}(E) = \mu_{\epsilon H}(\mathcal{G}_{\epsilon}).$$
It follows that
{\small\begin{equation}\label{eq:slope-max}
\begin{split}
\mu_{\max}(E) + &\sum_{k <n-1}\epsilon^{n-1-k}  \sum_{i_1, \cdots, i_k}\mu_{i_1, \cdots, i_k; ~H}(\mathcal{G}_0)  \\
&\leq \mu_{\epsilon H, \max}(E)= \mu_{\epsilon H}(\mathcal{G}_{\epsilon})
\leq \mu_{\max}(E) + \sum_{k <n-1}\epsilon^{n-1-k}  \sum_{i_1, \cdots, i_k}\mu_{i_1, \cdots, i_k; ~H,~ \max}(E).
\end{split}
\end{equation}}
Applying the above inequality to the $F^{k*}E$'s, as $E$ satisfies fdHN with respect to the $\mu$- and $\mu_{i_1, \cdots, i_k; ~H}$-slopes, we can show that
$$\lim_{\epsilon \to 0} L_{\epsilon H, \max}(E) = L_{\max}(E).$$
Similarly we can show $\lim_{\epsilon \to 0} L_{\epsilon H, \min}(E) = L_{\min}(E)$.

We generalize the results of \cite[Corollary 6.2]{Lan04} as follows.
\begin{lem}\label{lem:slope}
Let $X$ be a smooth projective variety of dimension $n$, $E$ a torsion free coherent sheaf of rank $r$ and $D_1, \cdots, D_{n-1}$ nef divisors on $X$. Assume that the intersection $D_1\cdot D_2 \cdots D_{n-1}$ is nontrivial and consider the $(D_1, \cdots, D_{n-1})$-slope. Then

(i) $\alpha(E) \leq \frac{r-1}{p}[L_{\mathrm{max}}(\Omega_X^1)]_{+}$ where  $[-]_{+}:=\max\{-, 0\}$;

(ii) if $E$ is semistable then
$$L_{\mathrm{max}}(E) - L_{\mathrm{min}}(E) \leq \frac{r-1}{p}[L_{\mathrm{max}}(\Omega_X^1)]_{+};$$

(iii) if $E$ is semistable and has rank two then
$$L_{\mathrm{max}}(E) - L_{\mathrm{min}}(E) \leq \frac{1}{p}[\mu_{\mathrm{max}}(\Omega_X^1)]_{+}.$$
\end{lem}
\begin{proof}
The assertions (i) and (ii) have been proved in \cite[Corollary 6.2]{Lan04} when the $D_i$'s are ample, and in general we can prove them by perturbing the $D_i$'s to be ample divisors and taking the limit.
To prove (iii), we may assume $E$ is not strongly semistable, since $E$ satisfies fDHN, there exists a minimal $k\geq 1$ such that $\mu_{\max}(F^{k*}E) = p^kL_{\max}(E)$, thus if  $E_1^k$ is the maximal destablizing subsheaf of $F^{k*}E$ then $E_1^k$ cannot descend to a subsheaf of $F^{(k-1)*}E$.
As a consequence, the canonical connection induces a non-trivial $\mathcal{O}_X$-linear homomorphism $E_1^k \to (F^{k*}E/E_1^k) \otimes \Omega_X^1$. Since rank~$F^{k*}E/E_1^k = 1$, we can show the desired inequality by
$$p^kL_{\mathrm{max}}(E) = \mu(E_1^k) \leq \mu_{\max}((F^{k*}E/E_1^k )\otimes \Omega_X^1) = p^kL_{\mathrm{min}}(E) + \mu_{\mathrm{max}}(\Omega_X^1).$$
\end{proof}


\subsection{Slopes of $1$-foliations}\label{sec:1-fol}
Let $X$ be a smooth variety over an algebraically closed field $k$ with $\mathrm{char}~ k=p >0$. Recall that a \emph{$1$-foliation} is a saturated subsheaf $\mathcal{F} \subset T_X$ which is involutive (i.e., $[\mathcal{F}, \mathcal{F}] \subset \mathcal{F}$) and $p$-closed (i.e., $\xi^p \in \mathcal{F}, \forall \xi \in \mathcal{F}$).
A $1$-foliation $\mathcal{F}$ induces a finite purely inseparable morphism (cf. \cite{Ek87})
$$\pi: X \to Y = X/\mathcal{F} = \mathrm{Spec} (Ann(\mathcal{F}) := \{a \in \O_X| \xi(a) = 0, \forall \xi \in \mathcal{F}\}),$$
and if $\mathcal{F}$ is a locally free subsheaf of $T_X$, then $Y$ is smooth and
$$K_{X} \sim \pi^*K_{Y} + (p-1)\det \mathcal{F}|_{X}.$$
For a subsheaf of $T_X$ be to a $1$-foliation, we have the following criterion.
\begin{lem}\textup{(\cite[Lemma 1.5]{Lan15})}\label{lem:c-fol}
Let $\mathcal{F}$ be a saturated $\O_X$-submodule of $T_X$. If
$$\mathrm{Hom}_X(\wedge^2\mathcal{F}, T_X/\mathcal{F}) = \mathrm{Hom}_X(F_X^*\mathcal{F}, T_X /\mathcal{F}) = 0$$
then $\mathcal{F}$ is a 1-foliation.
\end{lem}

\begin{thm}\textup{(\cite[Theorem 2.1]{Lan15})}\label{rat-curve}
Let $L$ be a nef $\mathbb{R}$-divisor on a normal projective variety $X$. Let $f : C \to X$ be a non-constant morphism
from a smooth projective curve $C$ such that $X$ is smooth along $f(C)$. Let $\mathcal{F} \subseteq T_X$ be a
$1$-foliation, smooth along $f(C)$. Assume that
$$c_1(\mathcal{F})\cdot C > \frac{K_X\cdot C}{p-1}.$$
Then for every $x \in f(C)$ there is a rational curve $B_x \subseteq X$ passing through $x$ such that
$$L \cdot B_x \leq 2\dim X \frac{pL \cdot C}{(p-1)c_1(\mathcal{F})\cdot C-K_X\cdot C}.$$
\end{thm}

\begin{cor}\label{cor:bound}
Let $X$ be a normal projective variety of dimension $d$ and let $A$ be an ample divisor on $X$. Let $\rho: Z \to X$ be a smooth resolution of singularities. Assume that $K_X$ is a nef $\mathbb{Q}$-Cartier divisor with Caritier index $n_0$ and numerical dimension $\nu(K_X) = l$. Set
$$D_1= \cdots D_{l} = \rho^*K_X,~D_{l+1} = \cdots D_{d-1} = \rho^*A.$$
Let $\mu$ be the $(D_1, \cdots, D_{d-1})$-slope.
Let $\mathcal{F} \subseteq T_Z$ be a saturated subsheaf.

(1) If $\mathcal{F}$ is a foliation with rank $r$ and $\mu(\mathcal{F}) > 0$
then either the nef dimension $n(K_X) \leq d-1$, or $K_X$ is big and then
$$r\mu(\mathcal{F}) = c_1(\mathcal{F})\cdot(\rho^*K_X)^{d-1}  \leq \frac{2pn_0d +1}{p-1}K_X^d.$$

(2) If $d=3$ and $\mu_{\mathrm{max}}(T_Z) = \mu(\mathcal{F}) >0$ then $\mathcal{F}$ is a foliation.
\end{cor}

\begin{proof}
(1) We will mimic the proof of \cite[Lemma 2.10]{XZ19}.

First assume $K_X$ is big. We may assume $A$ is sufficiently ample such that, for any sufficiently divisible positive integer $m$ the divisor $mK_X + A$ is also very ample on $X$ (see \cite{Ke08}). Take a curve $C_m$ from the intersection of $d-1$ general divisors in $\rho^*|mK_X +A|$. We can assume $\rho_*C_m$ is contained in the smooth locus of $K_X$ and $\mathcal{F}$ is smooth along $C_m$.
Then
\begin{align*}
K_Z \cdot C_m =K_Z \cdot (m\rho^*K_X + \rho^* A)^{d-1} =K_X^d m^{d-1} + o(m^{d-1})
\end{align*}
and
\begin{align*}
c_1(\mathcal{F}) \cdot C_m = c_1(\mathcal{F}) \cdot (m\rho^*K_X + \rho^* A)^{d-1} =c_1(\mathcal{F})\cdot(\rho^*K_X)^{d-1} m^{d-1} + o(m^{d-1}).
\end{align*}
If $c_1(\mathcal{F}) \cdot C_m \leq \frac{K_Z \cdot C_m}{p-1}$ for sufficiently large $m$, then the desired inequality holds. Otherwise, we set $L = \rho^*n_0K_X$ which is big, then we apply Theorem \ref{rat-curve} and obtain that for general closed point $z \in C_m$, there exists a rational curve $B_z$ passing through $z$ such that
{\small \begin{align*}1\leq L \cdot B_z \leq \frac{2dpL \cdot C_m}{(p-1)c_1(\mathcal{F})\cdot C_m- K_Z\cdot C_m}
&= \frac{2dpn_0K_X^d m^{d-1} + o(m^{d-1})}{((p-1)c_1(\mathcal{F})\cdot(\rho^*K_X)^{d-1} - K_X^d) m^{d-1} + o(m^{d-1})}\\
&= \frac{2dpn_0K_X^d}{(p-1)c_1(\mathcal{F})\cdot(\rho^*K_X)^{d-1} - K_X^d} + o(1).\end{align*}}
Taking $m\gg0$ shows the desired inequality.

If $K_X$ is not big, by similar argument as above we can get a family of rational curves $B_z$ passing through a very general point $z$ of $Z$ such that $\rho^*K_X \cdot B_z=0$ (here we use the fact that $L \cdot B_z$ is a nonnegative integer), which means the nef dimension $n(K_X) \leq d-1$.
\smallskip

(2) By the assumptions $\mu(T_Z) \leq 0$ and $\mu(\mathcal{F}) >0$, we have $\mu(T_Z/\mathcal{F}) <0$. By Lemma \ref{lem:c-fol}, to show $\mathcal{F}$ is a $1$-foliation it is enough to verify
$$\mathrm{Hom}_Z(\wedge^2\mathcal{F}, T_Z/\mathcal{F}) = \mathrm{Hom}_Z(F_Z^*\mathcal{F}, T_Z /\mathcal{F}) = 0.$$

It is trivial when $\mathrm{rank}(\mathcal{F}) = 3$. And when $\mathrm{rank}(\mathcal{F}) = 1$, the above two vanishings follow from $\wedge^2\mathcal{F} = 0$ and $ \mu(F_Z^*\mathcal{F}) = p \mu(\mathcal{F}) >  \mu(\mathcal{F}) > \mu_{\max} (T_Z /\mathcal{F})$ respectively.

Now assume $\mathrm{rank}(\mathcal{F}) = 2$. Since $\mu(\mathcal{F}) >0$, we have
$\mu(\wedge^2\mathcal{F}) =2 \mu(\mathcal{F}) > \mu(T_Z/\mathcal{F})$,
which infers the first vanishing $\mathrm{Hom}_Y(\wedge^2\mathcal{F}, T_Z/\mathcal{F}) = 0$.
If $F_Z^*\mathcal{F}$ is semistable, then the second vanishing holds. So we may assume $F_Z^*\mathcal{F}$ is not semistable, then the HN-filtration induces an exact sequence
\begin{equation}\label{ex-seq-1}
0\to \mathcal{F}_1 \to F_Z^*\mathcal{F} \to \mathcal{F}_2 \to 0.
\end{equation}
The canonical connection $\nabla_{\mathrm{can}}$ induces a non-trivial $\mathcal{O}_Z$-linear map $\mathcal{F}_1 \to \mathcal{F}_2 \otimes \Omega_Z^1$, which implies that
\begin{equation}\label{eq:5.1}
\mu(\mathcal{F}_1) - \mu(\mathcal{F}_2) \leq \mu_{\max}(\Omega_Z^1) = -\mu(T_Z /\mathcal{F}).
\end{equation}
By the exact sequence (\ref{ex-seq-1}), we have $\mu(\mathcal{F}_1) + \mu(\mathcal{F}_2) = 2p \mu(\mathcal{F})$. This equality minus the inequality (\ref{eq:5.1}) yields that
$$2\mu(\mathcal{F}_2) \geq 2p \mu(\mathcal{F}) + \mu(T_Z/\mathcal{F}) > 2\mu(T_Z/\mathcal{F})$$
which implies $\mathrm{Hom}_Z(\mathcal{F}_1, T_Z /\mathcal{F}) = \mathrm{Hom}_Z(\mathcal{F}_2, T_Z /\mathcal{F}) = 0$. Then we apply $\mathrm{Hom}_Z(-, T_Z /\mathcal{F})$ to the exact sequence (\ref{ex-seq-1}) and obtain the other vanishing
$\mathrm{Hom}_Z(F_Z^*\mathcal{F}, T_Z /\mathcal{F}) = 0$.
\end{proof}

\subsection{Proof of Theorem \ref{thm:my-ineq}} Before the proof let us recall the following results to estimate the discriminant $\Delta(E) = 2r c_2(E) - (r-1)c_1(E)^2$ where $r = \mathrm{rank}(E)$.
\begin{lem}\label{lem:bgmlv}
Let $X$ be a smooth projective variety of dimension $n$, $E$ a torsion free coherent sheaf of rank $r$ and $D_1, \cdots, D_{n-1}$ nef $\mathbb{R}$-Cartier divisors on $X$. Assume that $D_1\cdot D_2\cdots D_{n-1}$ is nontrivial and $F^{l*}E$ has a filtration
$$0 = E_0 \subset E_1 \subset E_2 \subset \cdots \subset E_m = F^{l*}E$$
with $F_i=E_i/E_{i-1}$ being torsion free. Let $r_i = \mathrm{rank}(F_i)$. Then

(1) $\Delta(E)\cdot D_2\cdots D_{d-1} =  \frac{1}{p^{2l}}D_2\cdots D_{d-1}\cdot(\sum\frac{r}{r_i}\Delta(F_i) -\sum r_ir_j(\frac{c_1(F_i)}{r_i}-\frac{c_1(F_j)}{r_j})^2)$;

(2) if $\Delta(F_i)\cdot D_2\cdots D_{d-1} \geq 0$ for $i=1,\cdots,m$, then
$$(D_1^2 \cdot D_2\cdots D_{d-1}) \cdot (\Delta(E)\cdot D_2\cdots D_{d-1}) \geq -r^2(\frac{\max_i\{\mu(F_i)\}}{p^l}- \mu(E))(\mu(E) - \frac{\min_i\{\mu(F_i)\}}{p^l});$$

(3) $(D_1^2 \cdot D_2\cdots D_{d-1}) \cdot (\Delta(E)\cdot D_2\cdots D_{d-1})\geq -r^2(L_{\mathrm{max}}(E)- \mu(E))(\mu(E) - L_{\mathrm{min}}(E))$.
\end{lem}
\begin{proof}
We refer the reader to \cite[p. 263]{Lan04} for the proof.
\end{proof}

\begin{proof}[Proof of Theorem \ref{thm:my-ineq}] Let $\mu$ denote the slope with respect to $D_1=D_2=\rho^*K_X$. Then
$$\mu(T_Z) = -\mu(\Omega_Z^1) = \frac{K_X^3}{3}.$$
We separate two cases according to the stability of the tangent bundle $T_Z$.

\smallskip

Case (i) $T_Z$ is semistable. By Lemma \ref{lem:slope} we have
$$L_{\mathrm{max}}(\Omega^1_Z) - L_{\mathrm{min}}(\Omega^1_Z) \leq \frac{2}{p}L_{\mathrm{max}}(\Omega^1_Z)$$
thus $L_{\mathrm{min}}(\Omega^1_Z) \geq (1-\frac{2}{p})L_{\mathrm{max}}(\Omega^1_Z) \geq 0$. Then we can show that $$L_{\mathrm{max}}(\Omega^1_Z) \leq 3\mu(\Omega^1_Z) = K_X^3.$$
Applying Lemma \ref{lem:bgmlv} (3), we obtain
$$6c_2(Z) \cdot \rho^*K_X - 2K_X^3 \geq -9\cdot\frac{2}{3}\frac{1}{3}K_X^3= -2K_X^3$$
which is equivalent to that $c_2(Z) \cdot \rho^*K_X \geq 0$.
\smallskip

Case (ii) $T_Z$ is not semistable. Consider the the HN-filtration
$$0 = \mathcal{F}_0 \subsetneqq \mathcal{F}_1 \subsetneqq \mathcal{F}_2 \subsetneqq \mathcal{F}_m = T_Z$$
where $2\leq m \leq 3$, and set $F_i = \mathcal{F}_i/\mathcal{F}_{i-1}$ and $r_i =\mathrm{rank}(F_i)$.

Case (ii-1) $m=3$. Then for each $i=1,2,3$, we have $\mathrm{rank}(F_i) =1$, and $\mu(F_1) + \mu(F_2) + \mu(F_3) = -K_X^3$. Applying  Corollary \ref{cor:bound}, we always have $\mu(F_1) \leq \frac{6n_0p+1}{p-1}K_X^3$.

If $\mu(F_2) \geq 0$, then $\mu(F_3) = -K_X^3 - (\mu(F_1) + \mu(F_2)) <0$ and $\mu(\wedge^2\mathcal{F}_2) = \mu(F_1) + \mu(F_2) >0$. It follows that
$$\mathrm{Hom}_Z(\wedge^2\mathcal{F}_2, T_Z /\mathcal{F}_2)= \mathrm{Hom}_Z(F^*F_1, T_Z /\mathcal{F}_2) = \mathrm{Hom}_Z(F^*F_2, T_Z /\mathcal{F}_2) = 0.$$
Then we can apply $\mathrm{Hom}_Z(-, T_Z /\mathcal{F}_2)$ to the exact sequence
$$0\to F^*\mathcal{F}_1 \to F^*\mathcal{F}_2 \to F^*F_2 \to 0$$
and obtain $\mathrm{Hom}_Z(F^*\mathcal{F}_2, T_Z /\mathcal{F}_2) = 0$. By Lemma \ref{lem:c-fol}, we see that $\mathcal{F}_2$ is also a foliation, thus $2\mu(\mathcal{F}_2) \leq \frac{6n_0p+1}{p-1}K_X^3$ by Corollary \ref{cor:bound}. Then it follows that
$$\mu(F_3) = -K_X^3 - 2\mu(\mathcal{F}_2) \geq -K_X^3 -\frac{6n_0p+1}{p-1}K_X^3= -\frac{(6n_0+1)p}{p-1}K_X^3.$$
Applying Lemma \ref{lem:bgmlv} (2) we get
{\small \begin{equation}\label{ineq:my-pre}
\begin{split}
6c_2(T_Z)\cdot \rho^*K_X - 2K_X^3 &\geq -9(\mu(F_1)+ \frac{K_X^3}{3})(- \frac{K_X^3}{3} - \mu(F_3))/K_X^3 \\
&\geq -9(\frac{6n_0p+1}{p-1}K_X^3 + \frac{K_X^3}{3})(- \frac{K_X^3}{3}+ \frac{(6n_0+1)p}{p-1}K_X^3)/K_X^3 \\
&= -\frac{(18^2n_0^2 + 54n_0 +2)p^2 + (54n_0 +5)p + 2}{(p-1)^2}K_X^3.
\end{split}
\end{equation}}
This is equivalent to the desired inequality
\begin{equation}\label{ineq:my}c_2(T_Z)\cdot \rho^*K_X + \frac{(54n_0^2 + 9n_0)p^2 + (9n_0 +\frac{3}{2})p}{(p-1)^2}K_X^3 \geq 0.\end{equation}
It is worth mentioning that the equality in (\ref{ineq:my}) is attained only when
$$\mu(F_1) = \frac{6n_0p+1}{p-1}K_X^3, \mu(F_2)=0 ~\mathrm{and}~\mu(F_3)=-\frac{(6n_0+1)p}{p-1}K_X^3.$$

If $\mu(F_2) <0$, then
$$\mu(F_3) = -K_X^3 - 2\mu(\mathcal{F}_2) > -K_X^3 - \mu(F_1) \geq (-K_X^3 -\frac{6n_0p+1}{p-1}K_X^3)= -\frac{(6n_0+1)p}{p-1}K_X^3,$$
and it is easy to verify the strict inequality in (\ref{ineq:my}) by the computation (\ref{ineq:my-pre}).

Case (ii-2) $m=2$. In this case $F_1 \cong \mathcal{F}_1$. We claim that
\begin{equation}\label{ineq:F1}r_1\mu(F_1) \leq \frac{6n_0p+1}{p-1}K_X^3.\end{equation} Indeed, if $\mu(F_1) \leq 0$ then this inequality automatically holds; otherwise, we can apply Corollary \ref{cor:bound}.
It follows from this claim that
\begin{equation}\label{ineq:F2}0 > \mu(T_Z) > \mu(F_2) =\frac{ -K_X^3 - r_1\mu(F_1)}{r_2} \geq  -\frac{(6n_0+1)p}{r_2(p-1)K_X^3}.\end{equation}

If both $F_1$ and $F_2$ are strongly semistable, then $\Delta(F_i) \cdot \rho^*K_X \geq 0$ holds for each $i=1,2$ (\cite[Thm. 0.1]{Lan04}), we can apply Lemma \ref{lem:bgmlv} (2) to verify the inequality (\ref{ineq:my}) by the computation (\ref{ineq:my-pre}).

Assume $F_i$ (one of $F_1, F_2$) is not strongly semistable. Then $\mathrm{rank} (F_i) =2$ and by Lemma \ref{lem:slope}
\begin{equation}\label{ineq:gap}
L_{\mathrm{max}}(F_i) - L_{\mathrm{min}}(F_i) \leq \frac{1}{p}[\mu_{\max}(\Omega^1_Z)]_{+}=- \frac{1}{p}\mu(F_2).
\end{equation}

If it is $F_1$ that is not strongly semistable, then $r_1=2,r_2=1$ and
\begin{equation}\label{eq:F_1}L_{\mathrm{max}}(F_1) + L_{\mathrm{min}}(F_1) = 2\mu(F_1)=  -K_X^3 - \mu(F_2).\end{equation}
The equation (\ref{ineq:gap}) plus (\ref{eq:F_1}) yields
$$L_{\mathrm{max}}(F_1) \leq  -K_X^3 + (1+\frac{1}{p})(-\mu(F_2)) < \frac{6n_0p+1}{p-1}K_X^3$$
where the ``$<$'' is due to (\ref{ineq:F2}). And similarly, the equation (\ref{eq:F_1}) minus (\ref{ineq:gap})
yields
$$L_{\mathrm{min}}(F_1) \geq  -K_X^3 - (1-\frac{1}{p})(-\mu(F_2)) \geq - 6n_0K_X^3.$$
Then considering the HN-filtration of $F^{l*}F_1$ for some sufficiently large $l$, which gives a refinement of the filtration
$F^{l*}F_1 \subseteq F^{l*}T_Z$, and applying Lemma \ref{lem:bgmlv} (2) we can show
\begin{align*}
6c_2(T_Z)\cdot \rho^*K_X - 2K_X^3 &\geq -9(L_{\mathrm{max}}(F_1) + \frac{K_X^3}{3})(- \frac{K_X^3}{3}- \min\{L_{\mathrm{max}}(F_1), \mu(F_2)\})/K_X^3 \\
&\geq -9(\frac{6n_0p+1}{p-1}K_X^3 + \frac{K_X^3}{3})(- \frac{K_X^3}{3}  +\frac{(6n_0+1)p}{p-1}K_X^3)/K_X^3,
\end{align*}
thus the desired inequality (\ref{ineq:my}) follows.

If it is $F_2$ that is not strongly semistable, then
$$L_{\mathrm{max}}(F_2) + L_{\mathrm{min}}(F_2) = 2\mu(F_2) = -K_X^3- \mu(F_1).$$
This equality plus (\ref{ineq:gap}) yields that
$$L_{\mathrm{max}}(F_2) \leq \mu(F_2) +  \frac{1}{2p}(-\mu(F_2)) < 0,$$
and thus
$$L_{\mathrm{min}}(F_2) > 2\mu(F_2) \geq -\frac{(6n_0+1)p}{(p-1)}K_X^3.$$
Similarly as in the previous case, we can show
\begin{align*}
6c_2(T_Z)\cdot \rho^*K_X - 2K_X^3 &\geq -9(\max\{\mu(F_1),L_{\mathrm{max}}(F_2)\}  + \frac{K_X^3}{3})(- \frac{K_X^3}{3}- L_{\mathrm{min}}(F_2) )/K_X^3 \\
&> -9(\frac{6n_0p+1}{p-1}K_X^3 + \frac{K_X^3}{3})(- \frac{K_X^3}{3}  +\frac{(6n_0+1)p}{(p-1)}K_X^3)/K_X^3
\end{align*}
which infers the desired inequality (\ref{ineq:my}).
\end{proof}

\section{Effectivity of pluricanonical maps for 3-folds}\label{sec:dim3}
Let $X$ be a minimal terminal threefold over an algebraically closed field $k$ of characteristic $p>0$. We aim to find natural numbers $M_1, M_2$ such that $S^0_{-}(X, K_X + mK_X) \neq 0$ if $m\geq M_1$, and that $S^0_{-}(X, K_X + mK_X)$ is birational if $m\geq M_2$. We conclude Theorem \ref{thm:eff-3folds} by separating the following two cases.

\subsection{Regular case}
In this case since $h^1(\O_X) - h^2(\O_X) \leq q(X)=0$ (\cite[Remark 9.5.15, 9.5.25]{FGA05}), we have $$\chi(\O_X) = h^0(\O_X) - (h^1(\O_X) - h^2(\O_X)) - h^3(\O_X) \geq  h^0(\O_X)  - h^3(\O_X).$$ We also assume that $X$ is Gorenstein, in particular it has only rational singularities.

\begin{lem}\label{lem:van-h2} If $(n-1)(p-1)\geq 6$ then $h^2(X, nK_X) =0$. \end{lem}
\begin{proof}We argue by contradiction. Suppose that $h^2(X, nK_X) \neq 0$.
By Serre duality we have $H^1(X, (1-n)K_X) \neq 0$. Applying Fujita vanishing theorem, we can take a sufficiently ample divisor $H$ such that $$H^i(X, -H -mK_X) \cong H^{3-i}(X, H + (m+1)K_X)^{\vee} =0~\ \ \mathrm{for ~any}~i>0~\mathrm{and}~m\geq 0.$$
Since $X$ has at most isolated singularities, we may assume $H$ is a smooth hypersurface contained in the smooth locus of $X$.
Since $K_X|_H$ is nef and big, for any sufficiently large $m$ we have $h^1(\mathcal{O}_H(-mK_X|_H)) =0$ (see \cite[Thm. 4.3]{XZ19}).
By taking cohomology of the following exact sequence
$$0 \to \mathcal{O}_X(-mK_X -H) \to  \mathcal{O}_X(-mK_X) \to  \mathcal{O}_H(-mK_X|_H) \to 0$$
we get a long exact sequence and can find a number $m_0$ such that for any $m\geq m_0$, $H^1(X, -mK_X) =0$. As a consequence, there exists a natural number $e$ such that the pullback map of Frobenius map
$$F^*: H^1(X, p^e(1-n)K_X) \to H^1(X, p^{e+1}(1-n)K_X)$$
has nontrivial kernel. Applying \cite[Corollary 4.6]{XZ19}\footnote{The assertion is valid if the assumption that $X$ is smooth is replaced by that $K_X$ is Cartier.}, it must hold that $(p-1)p^e(n-1) - 1 \leq 4$. However, this contradicts to our assumption.
\end{proof}
Next we prove

\begin{lem}\label{lem:h0} Assume that $h^0(X, \omega_X) \leq 1$. Set $n_0(2)=13, n_0(3) =10, n_0(5) =9$ and $n_0(p) = 8$ if $p \geq 7$. Then $h^0(X, n_0(p) K_X) \geq 15$ for $p=2,3,5$, and $h^0(X, n_0(p) K_X) \geq 10$ for $p \geq 7$.
\end{lem}
\begin{proof}
Since $h^3(\O_X) = h^0(X, \omega_X)  \leq 1$, $\chi(\O_X) \geq 0$. Let $\rho: Z \to X$ be a smooth resolution. By Theorem \ref{thm:my-ineq}, we can set $A(2) = 273$, $A(3) = 151$, $A(5) = 103$ and $A(p) = 90$ for $p\geq 7$ to make sure that $c_2(Z) \cdot \sigma^*K_X + A(p)K_X^3 \geq 0$.
Since $R\rho_*\mathcal{O}_X \cong \rho_*\mathcal{O}_X$, applying Riemann-Roch formula we have
\begin{align*}
&h^0(X, nK_X) + h^2(X, nK_X) \\
&\geq \chi(Z, \rho^*\mathcal{O}_X(nK_X)) \\
&= \frac{2n^3 - 3n^2}{12}(\rho^*K_X)^3+ \frac{n}{12}(\rho^*K_X)\cdot(K_Z^2 + c_2(Z)) + \chi(\mathcal{O}_Z)\\
& = \frac{2n^3 - 3n^2}{12}K_X^3 + \frac{n}{12}K_X\cdot(K_X^2 + \rho_*c_2(Z)) + \chi(\mathcal{O}_X) \\
& \geq \frac{2n^3 - 3n^2}{12}K_X^3 + \frac{n}{12}K_X\cdot(K_X^2 + \rho_*c_2(Z)) \geq \frac{n(2n^2 - 3n +1-A(p))}{12}K_X^3.
\end{align*}
Note that for $n \geq 6$, $h^2(X, nK_X) =0$ by Lemma \ref{lem:van-h2}. Then we can verify the lemma by direct computations.
\end{proof}

If $h^3(\mathcal{O}_X) \leq 1$ we let $n_0 = n_0(p)$ as in Lemma \ref{lem:h0}; otherwise we let $n_0=1$. Let $\phi_{n_0}: X \dashrightarrow \mathbb{P}^N$ denote the $n_0$-canonical map. We can do some blowup $\rho: Z \to X$ to assume that in the decompostion $\rho^*K_X = |H_Z| + E$, the movable part $|H_Z|$ has no base point and hence defines a morphism $\psi: Z \to \mathbb{P}^N$. Let $f: Z \to Y$ be the fibration arising from the Stein factorization of $\psi$. Then $H_Z = f^*H$ for an ample and free divisor $H$ on $Y$.

Case (1) $\dim\phi_{n_0} = 3$. Applying Theorem \ref{thm:bir-criterion-for-induction}, we can prove that for any $l >0$, $S_{-}^0(Z, K_Z + f^*3H + l\rho^*K_X) \neq 0$ and $S_{-}^0(Z, K_Z + f^*4H + l\rho^*K_X)$ is birational. Since $n_0\rho^*K_X \geq f^*H$, applying Proposition \ref{prop:F-stable-section} (i) and (iv) we only need to set $M_1 = 3n_0 +1$ and $M_2 = 4n_0 + 1$.

Case (2) $\dim\phi_{n_0} = 2$. Applying Theorem \ref{thm:eff-curve} and Theorem \ref{thm:bir-geo-gen}, since $K_X$ is Cartier, we have that $S_{-}^0(Z_{\eta}, K_{Z_{\eta}} + 2\rho^*K_X) \neq 0$, and $S_{-}^0(Z_{\eta}, K_{Z_{\eta}} + l\rho^*K_X)$ is birational for $l \geq 3$.
It follows by Theorem \ref{thm:bir-criterion-for-induction} that for any $l >0$, $S_{-}^0(Z, K_Z + f^*2H + l\rho^*K_X) \neq 0$ and $S_{-}^0(Z, K_Z + f^*3H + l\rho^*K_X)$ is birational. Then arguing similarly as in Case (1) we may set $M_1 = 2n_0 + 2$ and $M_2 = 3n_0 + 3$.

Case (3) $\dim\phi_{n_0} = 1$. In this case, $Y = \mathbb{P}^1$ since $q(Z) =0$. Since it suffices to show the result after a base field extension we may assume $k$ is uncountable. Let $F$ denote a general fiber of $f$ and let $\bar{F} = \rho_*(F)$.  Note that the Weil divisor $\bar{F}$ is Cartier in codimension two, and the divisor $K_{\bar{F}} \sim (K_X + \bar{F})|_{\bar{F}}$ is $\mathbb{Q}$-Cartier and Cartier in codimension one on $\bar{F}$. Let $G$ be a smooth resolution of $F$. By the adjunction formula we can write that
$$K_G \sim  K_{\bar{F}}|_G + E_2' - E_1'$$
where $E_1',E_2'$ are effective divisors without common components. Then $E_2'$ must be exceptional over $\bar{F}$, and we can prove that each its irreducible component is a rational curve by running a minimal model program over $\bar{F}$.


Case (3.1) $h^0(X, \omega_X) \leq 1$. We may assume $H_Z \sim r_0F$ where $r_0 = h^0(X, n_0K_X) -1$. Then by Lemma \ref{lem:h0}, we have that $r_0 > n_0$ and $n_0K_X \geq r_0\bar{F}$. We can write that
$$K_X= \frac{n_0}{r_0}K_X + (1-\frac{n_0}{r_0})K_X \sim_{\mathbb{Q}} \bar{F} + \bar{E} + (1-\frac{n_0}{r_0})K_X$$
where $\bar{E} \geq 0$, and we may assume $\bar{E}$ and $\bar{F}$ have no common component since $|\bar{F}|$ is movable. In turn we have
$$2K_X|_G \sim_{\mathbb{Q}} (K_X + \bar{F} + \bar{E} + (1-\frac{n_0}{r_0})K_X)|_G\sim_{\mathbb{Q}} K_G -  E_2' + E_1'+ (\bar{E} + (1-\frac{n_0}{r_0})K_X)|_G.$$
It follows that $(2K_X + E_2')|_G - K_G$ is big. To apply Corollary \ref{cor:eff-surface}, we set $D =K_X|_G$ and $r=2$, then obtain that $S^0_{-}(G, K_G + sK_X|_G) \neq 0$ for any integer $s\geq 5$, and $S^0_{-}(G, K_G + sK_X|_G)$ is birational for any integer $s\geq 9$. Applying Proposition \ref{prop:F-stable-section} (iv) and Corollary \ref{cor:bir-general-fiber}, analogous results hold for  $S^0_{-}(Z_{\eta}, K_{Z_{\eta}} + sK_X|_{Z_{\eta}})$. By Theorem \ref{thm:bir-criterion-for-induction} (i) we may set
$M_1 = n_0 + 5 $, and by the second part of (ii) we may set $M_2 = n_0 + 9$.

Case (3.2) $h^0(X, \omega_X) \geq 2$. Then $K_X \geq \bar{F}$. By similar argument of Case (3.1) we first obtain that $(3K_X + E_2')|_G - K_G$ is big, then apply Corollary \ref{cor:eff-surface} to prove $S^0_{-}(G, K_G + sK_X|_G) \neq 0$ for any integer $s\geq 6$, and $S^0_{-}(G, K_G + sK_X|_G)$ is birational for any integer $s\geq 11$.  Finally applying Theorem \ref{thm:bir-criterion-for-induction} we may set $M_1 = 7$ and $M_2 = 13$.
\smallskip

Remark that in Case (1), by taking a sub-linear system of $|n_0K_X|$ we can reduce us to the situation of Case (2), in practice the bounds of Case (2) are smaller. But we cannot reduce the former two cases to Case (3) because we do not have the relation $r_0 = h^0(X, n_0K_X) -1$ in Case (3.1).
\smallskip


\subsection{Irregular case} Let $a: X \to A$ be the Albanese map, let $\rho: Z \to X$ be a smooth resolution of singularities and let $f: Z \to Y$ be the fibration induced by the Stein factorization of $a\circ \rho: Z \to A$. Denote by $F$ the generic fiber of $f$. According to the relative dimension of $f$ we fall into three cases as follows.
\medskip

Case (1) $\dim F = 0$. We may set $M_1 = 3$ and $M_2=5$ by applying Theorem \ref{thm:bir-criterion-irr} (i,ii) and (iv) respectively.

Case (2) $\dim F = 1$. Since $K_X|_F = K_F$ has degree $\geq 2$, by Theorem \ref{thm:eff-curve} it follows that $S^0_{-}(F, K_F+ K_X|_F) \neq 0$ and $S^0_{-}(F, K_F+ 2K_X|_F)$ is birational. Then applying Theorem \ref{thm:bir-criterion-irr} we may set $M_1 = 1 +(1+1) = 3$ and $M_2=2 + 2 +2=6$.

Case (3) $\dim F = 2$. Since $X$ is smooth in codimension two, the generic fiber $F$ is regular and $K_X|_F \sim K_F$. By Theorem \ref{thm:eff-surface-nonclosed-field}, it follows that $S^0_{-}(F, K_F+ 5K_X|_F) \neq 0$ and $S^0_{-}(F, K_F+ 9K_X|_F)$ is birational; and if $p>2$ then $S^0_{-}(F, K_F+ 4K_X|_F) \neq 0$ and $S^0_{-}(F, K_F+ 7K_X|_F)$ is birational. Applying Theorem \ref{thm:bir-criterion-irr}, we may set $M_1=5+(1+5) = 11$ and $M_2=9+6+6 =21$; and if $p>2$, $M_1=4+(1+4)=9$ and $M_2=7 +5+5=17$.

\bibliographystyle{plain}

\end{document}